\documentclass[11pt]{article}

\usepackage{authblk}
\usepackage{amsfonts}
\usepackage{amsmath,amssymb}
\usepackage{amscd}
\usepackage{amsthm}
\usepackage{fancyhdr}
\usepackage{color}
\usepackage{hyperref} 
\usepackage{tikz-cd}
\usepackage[all,poly,web,knot]{xy}
\usepackage{fdsymbol}
\usepackage{float}%for figure positioning

\theoremstyle{plain}
\newtheorem{lem}{Lemma}[section]
\newtheorem{thm}[lem]{Theorem}

\newtheorem*{thmnonumber}{Theorem}
\newtheorem{cor}[lem]{Corollary} 
\newtheorem{prop}[lem]{Proposition} 
\newtheorem*{propnonumber}{Proposition}

\theoremstyle{definition}
\newtheorem{defi}[lem]{Definition}  
\newtheorem{ex}[lem]{Example} 
 
\theoremstyle{remark}
\newtheorem{rem}[lem]{Remark}

\newcommand{\cev}[1]{\reflectbox{\ensuremath{\vec{\reflectbox{\ensuremath{#1}}}}}}

\newcommand{\ZZ}{\ensuremath{\mathbb Z}}

\newcommand{\RR}{\ensuremath{\mathbb R}}

\newcommand{\CH}{\mathcal{H}}
\newcommand{\CS}{\mathcal{S}}
\newcommand{\CR}{\mathcal{R}}
\newcommand{\cB}{\mathcal{B}}

\newcommand{\cF}{\mathcal{F}}
\newcommand{\cG}{\mathcal{G}}

\newcommand{\cS}{\mathcal{S}}
\newcommand{\cO}{\mathcal{O}}

\newcommand{\cH}{\mathcal{H}}
\newcommand{\CK}{\mathcal{K}}

\newcommand{\g}{\mathfrak{g}}
\newcommand{\h}{\mathfrak{h}}

\newcommand{\CI}{C^\infty}
\newcommand{\CX}{\mathfrak{X}}

\newcommand{\SEC}{\Gamma}
\newcommand{\CF}{\mathcal{F}}
\newcommand{\CU}{\mathcal{U}}
\newcommand{\CG}{\mathcal{G}}
\newcommand{\Ho}{\mathcal{H}}

\newcommand{\fto}{\rightarrow}

\newcommand{\st}{\hspace{.05in}:\hspace{.05in}}
\newcommand{\y}{\hspace{.08in}\text{and} \hspace{.08in}}
\newcommand{\soutar}{\rightrightarrows}

\newcommand{\bt}{\mathbf{t}}                  %target
\newcommand{\bs}{\mathbf{s}}                  %source
\newcommand{\pd}[1]{\frac{\partial}{\partial #1}} %\pd{x}
\newcommand{\vX}{\mathfrak{X}} %vector fields

\title{Quotients of singular foliations\\
{and Lie 2-group actions}}

\author{Alfonso Garmendia and Marco Zambon\footnote{KU Leuven, Department of Mathematics, Celestijnenlaan 200B box 2400, BE-3001 Leuven, Belgium. \texttt{alfonso.garmendia@kuleuven.be}, \texttt{marco.zambon@kuleuven.be}}}

\begin{document}

\date{}

\maketitle

\begin{abstract}
Every singular foliation has an associated topological groupoid, called holonomy groupoid \cite{AndrSk}. In this note  
we exhibit some functorial properties of this assignment:  if a foliated manifold $(M,\mathcal{F}_M)$ is the quotient of a foliated manifold $(P,\mathcal{F}_P)$ {along} a surjective submersion with connected fibers, then the same is true for the corresponding holonomy groupoids. For quotients by a Lie group action, an {analogue} statement holds under suitable assumptions, yielding a Lie 2-group action on the holonomy groupoid.
	
\end{abstract}
\tableofcontents

 \section*{{Introduction}}
 \addcontentsline{toc}{section}{Introduction}

The space of leaves of a foliation is typically not smooth, and might fail to be Hausdorff. As a replacement for the leaf space one often takes a smooth group-like object canonically associated to the foliation, namely the holonomy Lie groupoid, and declares two Lie groupoids to model the same  space if they  are Morita equivalent. 

\emph{Singular foliations} extend the classical notion of foliation by allowing singularities. In this note we will adopt the definition of singular foliation that appears in \cite{AndrSk} and is inspired by the work of Stefan and Sussmann (among others) in the 1970's. It entails not only a smooth partition of the underlying manifold in immersed submanifolds of varying dimension, but it also contains information about the dynamics along the leaves, i.e. the ways that one can flow along them. This notion of singular foliation allows for Lie-theoretic constructions. The most prominent of them is the canonical assignment of a topological groupoid by Androulidakis-Skandalis \cite{AndrSk}, which is called \emph{holonomy groupoid} since in the case of (regular) foliations it recovers the holonomy Lie groupoid mentioned above.

 The assignment of the holonomy groupoid to a singular foliation satisfies  functoriality properties: under certain conditions, a map $\pi$ between foliated manifolds induces canonically a morphism of holonomy groupoids, which can be regarded as a replacement for the induced map of leaf spaces.
In this paper we prove properties of this assignment when  $\pi$ satisfies a surjectivity property, and thus can be regarded as a quotient map. {We do so motivated by the desire to establish to which extent the construction of the holonomy groupoid satisfies functorial properties.} {Since every holonomy groupoid has an associated a $C^*$-algebra \cite{AndrSk}, this is relevant also in non-commutative geometry,    and  the natural appearance of higher Lie groupoids in our work seems an interesting phenomenon in that context too.}

 We now outline the main results.
 
  \paragraph{\bf Statement of results}
 Let $\cF$ be a singular foliation on a manifold $P$, and  
let $\pi \colon P \to M$ be a {surjective} submersion with connected fibers.
Under mild invariance conditions, $\cF$ can be pushed forward along $\pi$ to a singular foliation $\cF_M$ on $M$. 
Since the foliation $\cF_M$ is obtained from $\cF$ by a quotient procedure, it is natural to wonder whether the 
holonomy groupoid $\Ho(\cF_M)$  is also quotient of the holonomy groupoid $\Ho(\cF)$. We show that this is always the case (see Thm. \ref{thm:sur.hol}):
\begin{thmnonumber}
The map $\pi$ induces a canonical {open} surjective morphism $$\Xi \colon \Ho(\cF)\to \Ho(\cF_M).$$
\end{thmnonumber}
\noindent
 We {emphasize} that this is a statement about (typically not source simply connected) topological groupoids. It is an analogue of the following fact in Lie groupoid theory: let $A, B$ be  integrable Lie algebroids,  and $\cG, \cH$ the source simply connected Lie groupoids integrating them.
Given a morphism of integrable Lie algebroids $\phi\colon A\to B$ which is fiberwise surjective and covers a surjective submersion between the manifolds of objects\footnote{Such a morphism is called a Lie algebroid fibration in \cite{MK2}.}, there is a unique Lie groupoid morphism $\Phi\colon \cG \to \cH$ integrating $\phi$, and further $\Phi$ is a\footnote{{$\Phi$ may fail to be a Lie groupoid fibration, see for instance Ex. \ref{ex:fail}.}}surjective submersion.

 We then refine the above result in the case of Lie group actions. That is, we suppose that $\pi \colon P \to M=P/G$ is the quotient map of the action of a Lie 
  group $G$   on $P$, which we assume to be  free, proper, and
   preserving the singular foliation $\cF$. The action lifts naturally to a $G$-action on $\Ho(\cF)$, but a simple dimension count shows that the quotient can not be isomorphic to $\Ho(\cF_M)$ in general. In \S\ref{sec:quotgrouppull} we show that 
when  $\cF$ contains the infinitesimal generators of the $G$-action,
there is a natural action of a semidirect product Lie group $G \rtimes G$ on $\Ho(\cF)$ -- not by Lie groupoid automorphisms -- with quotient $\Ho(\cF_M)$. 
Remarkably, this is a Lie 2-group action (see Thm.
\ref{thm:act.q.fol}). In other words:
\begin{thmnonumber}
When  $\cF$ contains the infinitesimal generators of the $G$-action, the induced morphism $\Xi$ is the quotient map of a Lie group action in the category of {topological} groupoids.
\end{thmnonumber}
\noindent {We expect the above conclusion to hold in greater generality, namely under a  regularity condition on the intersection of $\cF$ with the   foliation generated by the  $G$-action on $P$. We hope to address this in a future paper.}

 In the general case that $\cF$ does not necessarily contain the infinitesimal generators of the $G$-action, 
 {the fibers of $\Xi$ coincide with the orbits of a  \emph{groupoid} action on $\Ho(\cF)$, which we describe in Prop. \ref{cor:grpd.act.hol}.} 

We also obtain a canonical Lie 2-group action, whose orbits however may be smaller than the fibers of $\Xi$ (see Prop. \ref{prop:Lie.2.grpACTION}  and Cor. \ref{prop:Lie.2.grpFIBER}):

\begin{propnonumber} 
There is a canonical Lie ideal $\h$ of $\g$ which gives rise to a Lie 2-group $H \rtimes G$ and a Lie 2-group action on $\Ho(\cF)$,
whose orbits are contained  in the $\Xi$-fibers 
\end{propnonumber}

\noindent The concrete form of this Lie 2-group action is inspired   by the special case in which $\cF$ contains the infinitesimal generators of the $G$-action (hence $H=G$). Indeed in that case we recover  the Lie 2-group action given in \S\ref{sec:quotgrouppull}.

{We now return to the general set-up of the first theorem (in particular, the surjective submersion $\pi \colon P \to M$ does not necessarily arise from a Lie group action).
In \S\ref{section:groidfib} we address the question of when the open surjective morphism $\Xi \colon \Ho(\cF)\to \Ho(\cF_M)$ obtained in the first theorem is a fibration. This is useful because
when $\Xi$ is a fibration of Lie groupoids \cite{MK2}, it allows to describe the holonomy groupoid $\Ho(\cF_M)$  without knowing $\cF_M$. Namely, $\Ho(\cF_M)$ is the quotient of $\Ho(\cF)$ by a normal subgroupoid system; the latter -- as we explain just before Ex. \ref{ex:quotsim} -- is a set of data defined directly and explicitly in terms of $\Ho(\cF)$ and the projection $\pi\colon P\to M$. There is also a notion of fibration of topological groupoids \cite{BussMeyer}, which however appears to be less useful for the purpose of describing the quotient groupoid.}

{
We summarize as follows Ex. \ref{ex:fail}, Prop. \ref{prop:pullbackfibr}, Prop. \ref{prop:123fibr} and Prop. \ref{prop:norsubG}:
\begin{propnonumber} 
\begin{itemize}
\item[i)] The  open surjective morphism $\Xi \colon \Ho(\cF)\to \Ho(\cF_M)$ generally fails to be a fibration of topological groupoids. 
In the smooth case, $\Xi$ generally fails to be a fibration of Lie groupoids.
\item[ii)] Suppose that $\cF$ is the pullback foliation of $\CF_M$. Then the morphism $\Xi$ is a fibration of topological groupoids. 
In the smooth case, $\Xi$ is a fibration of Lie groupoids.
\item[iii)] Suppose that $\pi$ is the quotient map of a free, proper $G$-action on $P$  which
   preserves $\cF$. Then the morphism $\Xi$ is a fibration of topological groupoids, provided a technical condition is satisfied.  In the smooth case, $\Xi$ is always a fibration of Lie groupoids.
\end{itemize}
\end{propnonumber} 
}

\paragraph{\bf Conventions and acknowledgements} {Given a groupoid $\cG$ with space of objects $P$, its source map is denoted by $\bs \colon \cG\to P$, its target map by $\bt \colon \cG\to P$, and its product (multiplication) $\cG {{}_{\bs}\!\times_{\bt}} \cG\to \cG$ is denoted by $\circ$.}

{We thank the referee for suggestions that improved the content of paper.}
 We acknowledge partial support by the long term structural funding -- Methusalem grant of the Flemish Government, the FWO under EOS project G0H4518N, the FWO research project G083118N (Belgium).

\section{{Quotients of foliations by surjective submersions}}
{
We start reviewing singular foliations in the sense of \cite{AndrSk} and the topological groupoids associated to them. As recalled in \S\ref{subsec:main}, given  a {surjective} submersion with connected fibers $P \to M$, an ``invariant'' singular foliation $\cF$ on $P$ induces  a singular foliation $\cF_M$ on $M$, which can be regarded as a quotient of the former.
The main statement of the paper is that the holonomy groupoid of $\cF_M$ is a quotient of the 
holonomy groupoid of $\cF$, see Thm. \ref{thm:sur.hol}.
In \S\ref{subsec:char} we give an explicit characterization of the quotient map when $\cF$ is a  pullback-foliation.}

\subsection{{Background on singular foliations and holonomy groupoids}}

We review first the notions of singular foliation and holonomy groupoid from the work \cite{AndrSk} by  Androulidakis-Skandalis.

\subsubsection*{{Singular foliations}}

\begin{defi}\label{def:singfol}
	A \textbf{singular foliation} on a manifold $P$ is a $\CI(P)$-submodule $\CF$
	{of the compactly supported vector fields}
	$\CX_c(P)$, closed under the Lie bracket and locally finitely generated.
	A \textbf{foliated manifold} is a manifold with a singular foliation.
\end{defi}
\begin{rem}\label{rem:loc.gen} Let $P$ be a manifold and $\CF$ a submodule of $\CX_c(P)$. Take an open set $U\subset P$ and consider
% the following module:
	$$\iota_U^{-1}\CF:=\{ X|_U \st X\in \CF \y \text{supp}(X)\subset U\}.$$
	The module $\CF$ is {\bf locally finitely generated} if for every point of $P$ there is an open neighborhood $U$ and finitely many vector fields $X_1,\dots,X_n\in \CX(U)$ such that $\iota_U^{-1}\CF$ is 
{$Span_{{C_c^{\infty}(U)}} \{	X_1,\dots, X_n\}$}.
\end{rem}

Any singular foliation gives rise to a singular distribution that satisfies the assumptions of the  Stefan-Sussmann theorem; therefore, it induces a partition of the manifold into immersed submanifolds called leaves.  
\begin{ex}
	{i) Given an involutive regular distribution $D\subset TP$, which corresponds to a regular foliation by the Frobenius theorem, we obtain a singular foliation $\CF:=\Gamma_c(D)$.}
	
	{ii) If $A$ is a Lie algebroid over $P$ with anchor $\sharp\colon A\fto TP$, then $\sharp(\Gamma_c(A))$ is a singular foliation.}
\end{ex}

{The following  vector spaces  measure the regularity of a singular foliation at a given point.}
\begin{defi} Let $(P,\CF)$ be a foliated manifold and $p\in P$. Denote $I_p:=\{f\in \CI(P) \st f(p)=0\}$.  
	
	The {\bf tangent of $\CF$} at $p$ is $F_p:=\{X(p) \st X\in \CF\}\subset T_p P$.
	
	The {\bf fibre of $\CF$} at $p$ is $\CF_p:= \CF/I_p\CF$.
\end{defi}
If $\mathrm{dim}(F_p)$ is constant then $\CF$ is a regular foliation. If $\mathrm{dim}(\CF_p)$ is constant then $\CF$ is a {proyective} module and it is isomorphic to the sections {of} a vector bundle.

\begin{defi}\label{def:pullback}
	Let $(M,\CF_M)$ be a foliated manifold  and  $\pi:P\fto M$ a submersion.
Denote	 $$\pi^* \CF_M:=Span_{{C_c^{\infty}(P)}} \{X\circ \pi:X\in \CF_M\},$$
a submodule of sections of  the pull-back vector bundle $\pi^*TM$ over $P$.	
	The {\bf pullback foliation} of $\CF_M$ under $\pi$ is 	\[\pi^{-1}(\CF_M):=   (d\pi)^{-1}(\pi^*\CF_M),\]
	where the pre-image is taken with respect to 
	$d\pi\colon  \CX_c(P)\to {\Gamma(\pi^*TM)}, Y\mapsto d\pi (Y).$
The pullback foliation is a singular foliation on $P$ \cite[Prop. 1.10]{AndrSk}.	
\end{defi}

\begin{rem}
The original definition in \cite{AndrSk} is given in more generality, for any  smooth map $\pi$ transverse to $\CF_M$. When $\pi$ is a submersion,  we also have the following description:
\[\pi^{-1}(\CF_M)= Span_{{C_c^{\infty}(P)}} \left\{(d\pi)^{-1}(\{X\circ \pi:X\in \CF_M\})\right\},\]
	 i.e. the pullback foliation is generated by projectable vector fields whose projection lies in $\CF_M$.
\end{rem}

\begin{defi}\label{defi:globalhull} Given a submodule  $\CF$ of $\CX_c(P)$, its {\bf global hull}  is given by:
	$$\widehat{\CF}:=\{X\in \CX(P) \st fX\in \CF \hspace{.1in} \forall f\in \CI_c(P)\}.$$
\end{defi}
{Given a {surjective} submersion $\pi : P \to M$ (not necessarily with compact fibers) and a singular foliation $\cF_M$ on $M$, it may happen that }the set of projectable vector fields in $\pi^{-1}(\CF_M)$ consists just of the zero vector field, but by definition of pullback foliation, $\pi^{-1}(\CF_M)$ is the {$\CI_c(P)$-span} of projectable vector fields in $\widehat{\pi^{-1}(\CF_M)}$.

\subsubsection*{{Holonomy groupoids}}

 Singular foliations as in definition \ref{def:singfol} contain more information than just the underlying partition of $P$ into leaves, {since they carry ``dynamics'' on $P$}. This extra information was used in \cite{AndrSk} to define the holonomy groupoid via the following ``building blocks'':
 
 \begin{defi}
 	Given a foliated manifold $(P,\CF)$, a \textbf{bisubmersion} for $\CF$ is a triple $(U,\bt,\bs)$ where $U$ is a manifold and  $\bt\colon U\fto P$, $\bs\colon U\fto P$ are submersions, such that:
 	$$\bs^{-1}(\CF)=\bt^{-1}(\CF)=\SEC_c(\ker(d\bs))+\SEC_c(\ker(d\bt)).$$
 \end{defi}

\begin{ex}
	Let $\CG\soutar P$ be a Lie groupoid and $\CF_\CG$ the singular foliation on $P$ given by the Lie algebroid of $\CG$. Any {Hausdorff} open set $U\subset \CG$,
together with $\bt|_U$ and $\bs|_U$, is a bisubmersion for $\CF_\CG$. In particular, if $\CG$ is a Hausdorff Lie groupoid, then it is a  bisubmersion.
\end{ex}

The following proposition, proven in \cite[\S 2.3]{AndrSk}, assures the existence of bisubmersions at any given point $p_0\in P$.  
\begin{prop}\label{prop:path.hol.bi}
	Given $p_0\in P$, let $X_1,\dots X_k\in \CF$ be vector fields whose classes in the fibre $\CF_{p_0}$ form a basis. For $v=(v_1,\dots, v_k)\in \RR^k$, put $\varphi_v:=\exp(\sum_i v_i X_i)$, where $\exp$ denotes the time one flow.
	Put $W=\RR^k\times P$, $\bt(v,p)=\varphi_v(p)$ and $\bs(v,p)=p$.
	There is a neighborhood $U\subset W$ of $(0,p_0)$ such that $(U,\bt,\bs)$ is a bisubmersion.
\end{prop}

\begin{defi}\label{def:path.hol.bi}
	A bisubmersion as in Prop. \ref{prop:path.hol.bi}, when it has $\bs$-connected fibers, is called \textbf{path holonomy bisubmersion}.
\end{defi}

\begin{defi}\label{def:bisec.carry}
	Let $(P,\CF)$ be a foliated manifold, $(U,\bt,\bs)$ a bisubmersion and $u\in U$. Denote $x=\bs(u)$.
	\begin{itemize}
		\item A \textbf{bisection} at $u$ consists of a  $\bs$-section $\sigma\colon V\fto U$, {defined on an open subset $V\subset P$,} 
		{whose image   is transverse to  the fibers of $\bt$ and passes through $u$.}
		\item Given a diffeomorphism $f\colon P\supset V\fto V'\subset P$, a bisubmersion $(U,\bt,\bs)$ is said to \textbf{carry} $f$ at $u\in U$ if there exists a bisection $\sigma$ at $u$ such that $f=\bt\circ\sigma$.
	\end{itemize}
\end{defi}

\begin{defi}\label{def:inv.comp.mor.bi}
	Let $(P,\CF)$ be a foliated manifold, $(U_1,\bt_1,\bs_1)$ and $(U_2,\bt_2,\bs_2)$ be bisubmersions for $\CF$.
	\begin{itemize}
		\item The \textbf{inverse} bisubmersion of $U_1$ is $U_1^{-1}:=(U_1,\bs_1,\bt_1)$, obtained interchanging the source and target maps.
		\item The \textbf{composition} bisubmersion is $U_1\circ U_2:=(U_1 {}_{\bs_1} \!\times_{\bt_2} U_2, \bt_1, \bs_2)$.
		\item A \textbf{morphism} of bisubmersions is a map $\mu\colon U_1\fto U_2$ that commutes with the respective source and target maps. We will say that it is a local morphism if it is defined on an open set of $U_1$.
	\end{itemize}
\end{defi}
If there is a morphism of bisubmersions {$\mu\colon U_1\fto U_2$, then any local diffeomorphism carried by $U_1$ at $u\in U_1$ will be carried by $U_2$ at $\mu(u)$}.

\begin{defi}\label{def:eq.rel}
	Let $(P,\CF)$ be a foliated manifold, $(U_1,\bt_1,\bs_1)$ and $(U_2,\bt_2,\bs_2)$ be bisubmersions for $\CF$, $u_1\in U_1$ and $u_2\in U_2$. We say that $u_1$ is {\bf equivalent} to $u_2$ if there is a local morphism of bisubmersions $\mu\colon U_1\supset U'_1\fto U_2$ with $\mu(u_1)=u_2$.
\end{defi}

The previous definition gives an equivalence relation on any family of bisubmersions, {as becomes clear from the following useful proposition}.

\begin{prop}\label{prop:equiv2} Let $(P,\CF)$ be a foliated manifold, $(U_1,\bt_1,\bs_1)$ and $(U_2,\bt_2,\bs_2)$ be bisubmersions for $\CF$, $u_1\in U_1$ and $u_2\in U_2$. Then $u_1$ is equivalent to $u_2$ if and only if $U_1$ and $U_2$ carry the same local diffeomorphism at $u_1$ and $u_2$ respectively.
\end{prop}
\begin{proof}
This is a direct consequence of  \cite[Cor. 2.11]{AndrSk}.
\end{proof}

\begin{defi}\label{def:atlas.bi}
	Let $\CU=\{U_i\}_{i\in I}$ a family of bisubmersions for $\CF$.
	\begin{itemize}
		\item A bisubmersion $U'$ is \textbf{adapted} to $\CU$ if for any $u'\in U'$ there is $u\in U\in \CU$ which is equivalent {to $u'$}. {A family of bisubmersions $\CU'$ is adapted to $\CU$ if any {bisubmersion} $U'\in \CU'$ is adapted to $\CU$.}
		\item We say that $\CU$ is an \textbf{atlas} if:
		\begin{enumerate}
			\item For all $p\in P$ there is a $U\in \CU$ that carries  {the identity diffeomorphism nearby $p$}.
			\item The inverse and finite compositions of elements of $\CU$ are adapted to $\CU$.
		\end{enumerate}
	\end{itemize}
\end{defi}

\begin{ex}\label{ex:UdeG}
	Let $\CG\soutar P$ be a Lie groupoid and $\CF_\CG$ its associated foliation. {Any cover $\CU_\CG:=\{ U_i\}_{i\in I}$ of $\CG$ {by open Hausdorff subsets} is an atlas of bisubmersions for $\CF_\CG$. In particular, if $\CG$ is Hausdorff, it is also an atlas. Any two atlases given by Hausdorff covers of $\CG$ are adapted to each other by the identity morphism on $\CG$.}
\end{ex}

\begin{prop}\label{prop:gpd.atlas}{\bf (Groupoid of an atlas)} Let $\CU$ be an atlas of bisubmersions for $\CF$. Denote  $$\CG(\CU):= \sqcup_{U\in \CU} U / \sim$$ where $\sim$ is the equivalence relation given in definition \ref{def:eq.rel}, and endow it with the quotient topology. There is a natural structure of open topological groupoid on $\CG(\CU)$ where the source and target maps are given by the source and target maps of the elements of $\CU$.
\end{prop}

The proof of the following statement can be found in   \cite[ \S 3.1]{MEsingfol} and in \cite{T.ALF}. 
\begin{prop}\label{prop:adapt}
Two atlases $\CU$ and $\CU'$ are adapted to each other if and only if their corresponding topological groupoids are isomorphic: $\CG(\CU)\cong \CG(\CU')$.  \end{prop}

{\begin{defi}\label{def:scon.bisub} {Let $\cS$ be a family of source connected bisubmersions such that $\cup_{U\in \CS} \bs(U)=M$, satisfying this condition:
for every $U\in \cS$ and $p\in \bs(U)$, there is an element $e_p\in U$ carrying 
 the identity {diffeomorphism}  nearby $p$. The atlas $\CU$ generated by $\cS$ is called a \textbf{source connected atlas}.}
	
\end{defi}}

{{For instance, a family $\cS$ of path holonomy bisubmersions for $\cF$ such that $\cup_{U\in\cS} \bs(U)=M$ generates a source connected atlas. It is called a \textbf{path holonomy atlas}.} Also, any source connected Hausdorff Lie groupoid giving rise to $\cF$ constitutes a source connected atlas.}

{The groupoid of a source connected atlas is source connected, and one can show \cite{T.ALF} that it is the same\footnote{ {Prop. \ref{prop:adapt} then implies that all source connected atlases are adapted to each other}.} for all source connected atlases.}

\begin{defi}\label{def:hol.grpd}
	The \textbf{holonomy groupoid} $\Ho(\CF)$ of $\CF$ is the groupoid of any {source connected atlas}.  
\end{defi}

In particular the holonomy groupoid is an open source connected topological groupoid, which does not depend (up to isomorphism)  on the choice of {source connected atlas}.

\begin{rem}\label{rem:hol.regfol}
Let $\CG\soutar P$ be a source connected Lie groupoid and $\CF_{\CG}$ its associated foliation.  Then, as we now explain, the holonomy groupoid $\Ho(\CF_\CG)$ is a quotient of $\CG$. 

Let $\cS_{\CG}$ be a Hausdorff source connected cover of a neighborhood of the identity bisection in $\CG$ and $\CU_{\CG}$ the atlas generated by $\cS_{\CG}$. {Then} {$\CU_{\CG}$ is a {source connected atlas}, hence} $\CG(\CU_\CG) \cong \Ho(\CF_\CG)$. One can show that $\CU_{\CG}$ is adapted to the atlas given in example \ref{ex:UdeG}, and therefore that {$\CG(\CU_\CG)\cong \CG/\sim$,  where the latter equivalence relation identifies two points when they carry the same local diffeomorphism. Consequently, $\Ho(\CF_\CG)\cong \CG/\sim$.}

 {When $\CF$ is a regular foliation, $\Ho({\CF})$ agrees with the classical notion of holonomy groupoid of a regular foliation.} Indeed,
by the Frobenius theorem there is a Lie algebroid $D\subset TP$ such that $\CF=\SEC_c(D)$.  {The monodromy groupoid $\Pi(\CF)$, consisting  of homotopy classes of paths in the leaves of $\CF$, is a source connected Lie groupoid  integrating $D$}. 
{Therefore $\Ho({\CF})\cong (\Pi(\CF)/\sim)$, and the latter is the well-known holonomy groupoid of a regular foliation.} 
\end{rem}

 \begin{rem}
{A singular foliation $\cF$ on $M$ is called {\bf projective} if there is a vector bundle $E$ such that $\cF\cong \Gamma_c(E)$  as $C^{\infty}(M)$-modules. These are exactly the singular foliations for which the holonomy groupoid $H(\cF)$ is a Lie groupoid. The class of projective foliations contains the regular foliations as a proper subclass.}
\end{rem}

\subsection{{The main theorem}}\label{subsec:main}

We   reproduce \cite[Lemma 3.2]{AZ1}, about quotients of foliated manifolds. 

\begin{prop}\label{prop:submfol} Let $\pi : P \to M$ be a {surjective} submersion with connected fibers. Let $\cF$ be a singular foliation on $P$, such that $\Gamma_c(\ker d\pi) \subset \cF$. Then there is a unique singular foliation  $\cF_M$ on $M$
with $\pi^{-1}(\cF_M) = \cF$. 
\end{prop}

The following theorem  is our main result and  will be proven in  Appendix \ref{app:thmXi}.

{
\begin{thm}\label{thm:sur.hol}
		Let $\pi : P \to M$ be a {surjective} submersion with connected fibers. Let $\cF$ be a singular foliation on $P$, such that 
		\begin{equation}\label{eq:bracketpres}
[\Gamma_c(\ker d\pi),\CF] \subset \Gamma_c(\ker d\pi)+\CF.
\end{equation}
Denote by $\CF_M$ the singular foliation on $M$ obtained from
$\CF^\text{big}:=\Gamma_c(\ker d\pi)+\CF$
as in Prop. \ref{prop:submfol}. 	Then there is a canonical,  {open,} surjective morphism of topological groupoids	
$$\Xi \colon \Ho(\CF)\fto \Ho(\CF_M)$$
		covering $\pi$.
\end{thm}
{This  should be interpreted as follows.
The singular foliation $\cF_M$ is obtained from $\cF$ by a quotient procedure ({more precisely $\cF_M=\pi_*\cF$, see Lemma \ref{lem:projgen2}
 ii))}. The theorem states that   the same is true for the respective holonomy groupoids.
}

\begin{rem}\label{rem:S}
{We now give a characterization of the morphism $\Xi$.
By Lemma \ref{lem:projgen2} i) and Cor. \ref{cor:adpt}, any source connected atlas $\CU$ for $\CF$ satisfies that $\pi \CU:=\{(U,\pi \circ\bt, \pi\circ\bs) \st U\in \CU\}$ is an atlas equivalent to a path holonomy atlas  for $\CF_M$.
The map $\Xi$ is characterized by} 
\begin{equation}\label{eq:Xidescr}
\Xi([u])=[u]_M
\end{equation}
 {for all $u\in \CU$, where $[u]_M$ is the  class of $u\in (U,\pi \circ\bt, \pi\circ\bs)$, a bisubmersion for $\CF_M$.}
\end{rem}

\begin{rem}\label{rem:ss} {When $\CH(\CF)$ and $\CH(\CF_M)$ are Lie groupoids, $\Xi$ is a surjective submersion (cf. Rem. \ref{rem:appsub})}
\end{rem}

{For regular foliations, the morphism $\Xi$ admits a familiar description.}

{\begin{prop} [{\bf Regular foliations}] \label{prop:regXi} When both $\cF$ and $\cF_M$ are regular foliations, the morphism $\Xi\colon \Ho(\CF)\fto \Ho(\CF_M)$ can be easily described by
$$\Xi([\gamma]_{hol})=[\pi\circ \gamma]_{hol},$$
for each curve $\gamma\colon[0,1]\fto P$ inside a leaf of $\CF$. {Here $[-]_{hol}$ denotes  holonomy classes.}
\end{prop}}

 \begin{proof} 
 {We may assume that $\Ho(\CF)$ is Hausdorff (when not, one needs to argue using a Hausdorff cover of it as in Rem. \ref{rem:hol.regfol}).}
 {The {Lie} groupoids $(\Ho(\CF),\bt,\bs)$ and $(\Ho(\CF_M),\bt_M,\bs_M)$ are source connected atlases for $\CF$ and $\CF_M$ respectively. 
The map $\widehat{\pi}\colon \Ho(\CF)\fto \Ho(\CF_M);[\gamma]_{hol}\mapsto[\pi\circ \gamma]_{hol}$ is a submersion, {since it is a Lie groupoid morphism integrating the fiber-wise surjective Lie algebroid morphism $\pi_*$}. This implies that $(\Ho(\CF),\bt_M \circ \widehat{\pi},\bs_M \circ \widehat{\pi})$  is a bisubmersion for $\CF_M$, by \cite[Lemma 2.3]{AndrSk}. Notice that the latter triple equals $(\Ho(\CF),\pi\circ \bt,\pi\circ \bs)$, which hence is a bisubmersion.}
	 	
{Using that $\pi$ has   connected fibers we have that $(\Ho(\CF),\pi\circ \bt,\pi\circ \bs)$ is a source connected atlas for $\CF_M$, and by Rem. \ref{rem:S} we can thus compute the map $\Xi$ as follows:}
$$\Xi([\gamma]_{hol})=\widehat{\pi}([\gamma]_{hol})=[\pi\circ\gamma]_{hol},$$
where the first equality holds by {eq. \eqref{eq:Xidescr}} and the fact that
  $\widehat{\pi}\colon (\Ho(\CF),\pi\circ \bt,\pi\circ \bs)\fto (\Ho(\CF_M),\bt_M,\bs_M)$ is a morphism of bisubmersions.
\end{proof}

We present an example for Thm. \ref{thm:sur.hol} where $\cF$ is a regular foliation  and $\CF_M$ is a genuinely singular foliation. Notice that the holonomy groupoid of the former foliation has discrete isotropy groups, whereas for the latter the isotropy groups are not all discrete.

\begin{ex}\label{ex:non.sing.q}
{
Consider the cylinder $P:=S^1\times \RR$ with coordinates $(\theta,y)$ and
 the   regular foliation\footnote{The foliation $\cF$ is the quotient by the natural $\ZZ$-action of the foliation on the $x$-$y$-plane whose leaves are given by $graph(e^{x+c})$  on the open upper plane, $graph(-e^{x+c})$ on the open lower plane (with $c$ varying through all real numbers), and the line $\{y=0\}$.} $\cF$  given by the  integral curves of the nowhere vanishing vector field $X:=\pd{\theta}+y\pd{y}$.
The circle $U(1)$ acts on the cylinder $P$ by rotations of the first factor, preserving the foliation $\cF$. The singular foliation $\CF^\text{big}$ on $P$ has three leaves (two open leaves, separated by the middle circle).
The quotient map $$\pi\colon P=S^1 \times \RR \to  M:=P/U(1)\cong \RR$$ is the second projection.  On the quotient, the induced foliation is $\cF_M=\langle y\pd{y}\rangle$, a genuinely singular foliation.}

{For the holonomy groupoids, we have $\Ho(\cF)=\RR\times P$, the transformation groupoid of the action of the Lie group $\RR$ on $P$ by the flow of $X$, which reads $\phi_t(\theta \text { mod }2\pi, y)=(\theta +t \text { mod }2\pi,e^ty)$. Further  $\Ho(\cF_M)=\RR\times M$, the transformation groupoid of the action of the Lie group $\RR$ on $M$ by the flow of $y\pd{y}$, which reads $\phi_t(y)=e^ty$. This follows from \cite[Ex. 3.7 (ii)]{AZ1}. The canonical surjective morphism of Thm. \ref{thm:sur.hol} is $$\Xi \colon \RR\times P\to \RR\times M,\;\; (t,p)\mapsto(t,\pi(p)).$$
This can be seen 
   from   {eq. \eqref{eq:Xidescr}, since the vector field $X$ $\pi$-projects to $y\pd{y}$.} 
Notice that, at points $S^1\times \{0\}$, the isotropy groups of $\Ho(\cF)$ are discrete, as for all regular foliations, while the  isotropy group of $\Ho(\cF_M)$ at the point $0\in M$ is isomorphic to $\RR$. 
}
\end{ex}

\begin{figure}[h]\label{fig:moebius}
	\centering
	\scalebox{.3}{\includegraphics{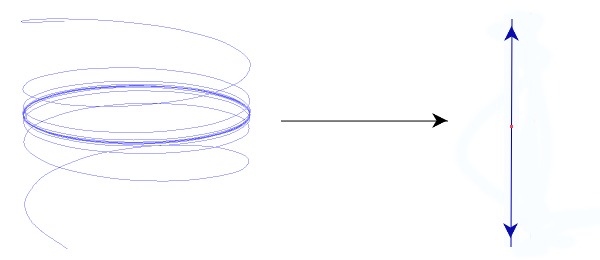}}
	\caption{The foliated manifolds in Example \ref{ex:non.sing.q}.}
\end{figure}

\subsection{A characterization of the quotient map for pullback-foliations}\label{subsec:char}

We make the map $\Xi$ in Thm. \ref{thm:sur.hol} more explicit in the special case that $\Gamma_c(\ker d\pi) \subset \CF$. {In this case $\cF=\CF^{big}:=\pi^{-1}\CF_M$ is the pullback of $\cF_M$ by $\pi$.}

We will need \cite[Thm. 3.21]{MEsingfol}, stated as follows:

\begin{thm}\label{thm:Homeo}
	Given a foliated manifold $(M,\CF_M)$ and a surjective submersion with connected fibers $\pi\colon P\fto M$, {there is a canonical isomorphism} $$\varphi\colon \Ho(\pi^{-1}(\CF_M))
	\xrightarrow{\sim}
	 \pi^{-1}(\Ho(\CF_M)),$$
{where the r.h.s.  denotes the pullback groupoid 	 $P{}_\pi \!\times_\bt \Ho(\CF_M) {}_\bs \!\times_\pi P$.}
\end{thm}

\begin{rem}\label{rem:varphi}
 Let $\CU$ be a path holonomy atlas for $\CF_M$. Then {$\pi^{-1}\CU:=\{\pi^{-1}U:U\in \CU\}$, where}
$$\pi^{-1}U:=P{}_\pi \!\times_\bt U {}_\bs \!\times_\pi P,$$ is a {source connected} atlas for $\pi^{-1}(\CF_M)$, see \cite{MEsingfol}. We describe the   isomorphism $\varphi$  by $[(p,u,q)]\mapsto (p,[u],q)$.
\end{rem}

Our alternative description of the map $\Xi$  is as follows:

\begin{prop}\label{prop:normalq}
	Let $\pi \colon P \to M$ be a {surjective} submersion with connected fibers. Let $\cF$ be a singular foliation on $P$, such that $\Gamma_c(\ker d\pi) \subset \CF$. Denote by $\CF_M$ the unique singular foliation on $M$ such that $\pi^{-1}(\CF_M)=\CF$. 
Under the canonical isomorphism $\varphi\colon \Ho(\CF) \xrightarrow{\sim} \pi^{-1}(\Ho(\CF_M)) $  given in Thm. \ref{thm:Homeo}, the following two morphisms coincide:
\begin{itemize}
\item the morphism $\Xi\colon \Ho(\CF)\fto \Ho(\CF_M)$    given by Thm. \ref{thm:sur.hol},
\item the second projection $pr_2\colon \pi^{-1}(\Ho(\CF_M))=P\times_M \Ho(\CF_M) \times_MP\fto \Ho(\CF_M)$.
\end{itemize}
\end{prop}

\begin{proof}
{Fix a path holonomy atlas $\CU_M$ for $\CF_M$. By Rem. \ref{rem:varphi} the family $\CU:=\pi^{-1}\CU_M$
is a {source connected} atlas for $\CF=\pi^{-1}(\CF_M)$. Moreover, one can show that for all $U\in \CU_M$ the triple $(\pi^{-1}(U),\pi\circ \bt,\pi\circ \bs)$ is a bisubmersion for $\CF_M$ .}
	
 We want to show that $\Xi\circ \varphi^{-1}=pr_2$. To this aim, take any $(p,{\zeta},q) \in \pi^{-1}(\Ho(\CF_M))$ (so in particular  ${\zeta}\in \Ho(\CF_M)$).
Fix a representative  $u\in U \in \CU$ such that $[u]={\zeta}$, {then $\varphi^{-1}(p,{\zeta},q)=[(p,u,q)]$} where $(p,u,q)\in \pi^{-1}U$. It is sufficient to show that 
\begin{equation}\label{eq:show}
\Xi([(p,u,q)])=[u].
\end{equation}
Note that $\hat{\pi}\colon (\pi^{-1}U,\pi\circ \bt,\pi\circ \bs)\fto U; (p,u,q)\mapsto u$ is a {morphism of} bisubmersions for $\CF_M$, therefore $(p,u,q)\in(\pi^{-1}U,\pi\circ \bt,\pi\circ \bs)$ is equivalent to $u\in (U,\bt,\bs)$. Then {eq. \eqref{eq:show} holds}
by the characterization of $\Xi$ given {in Rem. \ref{rem:S}}.
\end{proof}

\section{{Lie $2$-group actions on holonomy groupoids}}
\label{sec:quotgrouppull}

{We start reviewing Lie 2-groups and Lie 2-group actions.
	In \S\ref{subsec:ex.lie.2.grp} we present an important special case of Thm. \ref{thm:sur.hol} in which the map $\Xi$   is the quotient map of a Lie 2-group action on $\Ho(\CF)$ (see Thm. \ref{thm:act.q.fol} and Prop. \ref{prop:afterthm:act.q.fol}). We will revisit this special case later on, in \S\ref{subsec:alternative}.}

\subsection{{Background on Lie 2-groups}}\label{subsec:Lie.2.grp}

In the sequel will need the notion of Lie 2-group, which we recall here.

\begin{defi} 
{	A \textbf{Lie 2-group} is a group in the category of Lie groupoids.}
\end{defi} 
{In other words,  a Lie 2-group is a Lie groupoid $\CG\soutar G$ such that $\CG$ and $G$ are Lie groups, so that the group multiplication and group inverse are Lie groupoid morphisms,
and the inclusion of the neutral elements is a Lie groupoid morphism.}

{\begin{rem} {Equivalently, a Lie 2-group is a groupoid in the category of Lie groups.}
\end{rem}}

\begin{ex}\label{ex:GH}
Let $G$ be a Lie group and $H\subset G$ a normal Lie subgroup. Then $H$ acts on $G$ by left multiplication, leading to the action Lie groupoid $H\times G\soutar G$. {In particular, the groupoid composition is $$(h_2,h_1  g)\circ(h_1,g)=(h_2 h_1, g).$$}
Note that its  space of arrows has a group structure, namely the semidirect product by the conjugation action $C_g(h)=g h g^{-1}$ {of $G$ on $H$}. Explicitly, {the group multiplication} is given by
$$(h_1,g_1)\cdot(h_2,g_2)=(h_1C_{g_1} (h_2), g_1g_2).$$
{We write $H\rtimes G$ for $H\times G$ endowed with this group structure.}

{One can check that $H\rtimes G\soutar G$ is a Lie 2-group.}
\end{ex}

\begin{rem}\label{rem:xm} 
For the sake of completeness, we provide the description of a Lie 2-group in full generality.
A crossed module of Lie groups consists of  Lie groups   $H$ and $G$, Lie group morphisms $C\colon G\fto \mathrm{Aut}(H); g\mapsto C_g$ and $\bt\colon H\fto G$ such that $\bt(C_g(h))=g\bt(h)g^{-1}$ and $C_{\bt(h)}(j)=hjh^{-1}$ for all $g\in G$ and $h,j\in H$.
There is a bijection between Lie 2-groups and crossed {modules} of Lie groups \cite{CrossMod}.
Given a  Lie 2-group $\CG\soutar G$, the associated crossed module is given by $G$, by $H:=\ker(\bs)$ (a normal subgroup of  $\CG$), by the restriction $\bt\colon H\fto G$  of the target map, and the Lie group morphisms $C\colon G\fto \mathrm{Aut}(H); g\mapsto C_g(h):=g h {g^{-1}}$. Then $\cG$ as a Lie group is isomorphic to the semidirect product of $G$ and $H$ by the action $C$, and as a Lie groupoid it is isomorphic to the transformation groupoid of the $H$-action on $G$ by left multiplication with $\bt(\cdot)$.
\end{rem}

\begin{defi}
	A \textbf{Lie 2-group action} is a  {group} action in the category of Lie groupoids. 
\end{defi}

Hence an action of a Lie 2-group $\CG\soutar G$ on a Lie groupoid $\cH\soutar P$ consists of    {group} actions of $\cG$ on $\cH$ and of $G$ on $P$ such that the action map
\[\begin{tikzcd}
\cG \times \cH \ar[d, shift right=.2em, swap,"\bt\times \bt"]\ar[d,shift left=.2em,"\bs\times \bs"] \ar[r] & \cH \ar[d, shift right=.2em, swap,"\bt"]\ar[d,shift left=.2em,"\bs"] \\
G\times P \ar[r] & P
\end{tikzcd}\]
is a Lie groupoid map.
{Notice that such an action is not by Lie groupoid automorphisms of $\cH$. Nevertheless, the following result holds:}

\begin{prop}\label{prop:lie2grp.q}
 {Consider a free and proper action $\star$ of the  Lie 2-group}  $ \CG \soutar G$ on a Lie groupoid $\cH\soutar P$. 
	
Then $\cH':= \cH/\CG$ and $M:= P/G$ are manifolds, and $\cH'\soutar M$ acquires a canonical  Lie groupoid structure. Further  the projection $\cH\to \cH'$ is  a surjective submersion and Lie groupoid morphism.
\end{prop}
Although we do not need the above result\footnote{One can prove Prop. \ref{prop:lie2grp.q}, {and also the stronger statement that the projection is a fibration (Def. \ref{def:fibrLie})}. To do so, one can follow \cite[\S2.4]{MK2}: define $R=\{(gp,p)\in P\times P \st p\in P \y g\in G\}$ and $\CR=\{((h,g)\star\xi, \xi)\in\cH\times \cH \st \xi\in \cH \y (h,g)\in \CG\}$, and show that $(\CR,R)$ is a smooth congruence for $\cH$.  Then  Thm. \ref{thm:norm.sub.sys} gives the desired conclusion. See \cite[\S 5.3]{T.ALF} for more details.}, we mention it here because it puts in perspective Thm. \ref{thm:act.q.fol} below.

\subsection{Lie 2-group action on the holonomy groupoid of a pullback foliation}\label{subsec:ex.lie.2.grp}

Fix  a foliated manifold $(P,\cF)$ and a free and proper action of a connected Lie group $G$ on $P$ preserving $\CF$. {We denote
the quotient map by $\pi\colon M \fto M/G$.}
We assume that the infinitesimal generators of the $G$-action lie in the global hull $\widehat{\CF}$, i.e.  $\Gamma_c(\ker d\pi)\subset \CF$. This occurs exactly when $\cF$ is the pullback of $\cF_M$ by $\pi$, as in \S \ref{subsec:char}.

\begin{thm}\label{thm:act.q.fol}
Let $G$ be a connected Lie group acting freely and properly on a foliated manifold $(P,\CF)$. 
Assume {that $\Gamma_c(\ker d\pi)\subset \CF$.}  

Then there is a canonical Lie 2-group action\footnote{Here we use the term ``Lie 2-group action'' in a   loose way, since $\Ho(\CF)$ is generally not a Lie groupoid.}
 of $G\rtimes G\soutar G$  on the holonomy groupoid $\Ho(\CF)$.
\end{thm}

 Here $G\rtimes G\soutar G$ is endowed with the Lie 2-group structure of Ex. \ref{ex:GH}.

\begin{proof} We make use {of} the canonical isomorphism $\Ho(\CF) \cong \pi^{-1}(\Ho(\CF_{M}))$ given in  Thm. \ref{thm:Homeo}.

There is a canonical Lie $2$-group action of $G\rtimes G$ on   $\pi^{-1}(\Ho(\CF_{M}))$, extending the given action of $G$ on the base $P$,
 given by
\begin{equation}\label{eq:star}
(h,g)* (p,[v],q)=(hgp,[v],gq).
\end{equation}
It can be checked by computations that this defines a group action and groupoid morphism. Alternatively, we can use the isomorphism of Lie 2-groups to    $G \times G \soutar G$ (the pair groupoid, with product group structure) given by  $G\rtimes G\cong G\times G, (h,g)\mapsto (hg,g)$. Under this isomorphism,   
\eqref{eq:star} becomes 
$$ (G\times G)\times {\pi^{-1}(\Ho(\CF_{M}))}\to {\pi^{-1}(\Ho(\CF_{M}))},\;\; \left((h,g),(p,[v],q)\right) \mapsto (hp,[v],gq),$$
which is easily checked to   be a Lie $2$-group action.
\end{proof}
 
\begin{prop}\label{prop:afterthm:act.q.fol}
Assume the set-up of Thm. \ref{thm:act.q.fol}.

The orbits of the Lie 2-group action of $G\rtimes G\soutar G$ on $\Ho(\CF)$ are exactly the fibers of the canonical map $\Xi \colon \Ho(\CF)\fto \Ho(\CF_M)$. In particular, the quotient of $\Ho(\CF)$ by the action is canonically isomorphic to $\Ho(\CF_M)$.
	  
\end{prop}
\begin{proof}
{The formula   \eqref{eq:star} makes clear what the orbits are, and Prop. \ref{prop:normalq} shows that they agree with the $\Xi$-fibers.}
\end{proof}

\section{{Quotients of foliations by group actions: the general case}}
\label{sec:general}
In this section we consider the following set-up:
\begin{center}
	\fbox{
		\parbox[c]{12.6cm}{\begin{center}
				a foliated manifold $(P,\cF)$,\\
				a free and proper action of a connected Lie group $G$ on $P$ preserving $\cF$.\end{center}
	}}
\end{center}
Condition \eqref{eq:bracketpres} in Thm. \ref{thm:sur.hol} is satisfied. Hence we obtain an {open} surjective groupoid morphism 
\begin{equation}\label{eq:Xi}
\Xi\colon \Ho(\CF)\fto \Ho(\CF_M)  
\end{equation}
 covering the projection   $\pi \colon P\to M:=P/G$, where the latter is endowed with the foliation $\CF_M$ specified there.

Unlike the special case considered in \S \ref{subsec:ex.lie.2.grp},  {$\Gamma_c(\ker d\pi)$ may not be contained in $\CF$, hence} the $\Xi$-fibers are not the orbits of a Lie 2-group action in general. In this section we make two general statements about the $\Xi$-fibers.

In \S\ref{sec:groidorbits},  {we lift} the $G$-action on $P$ to an action on $\Ho(\CF)$ by groupoid automorphisms.  {Using this action, later in Proposition \ref{cor:grpd.act.hol}} we can characterize the fibers of $\Xi$ as the orbits of a \emph{groupoid}  action. 

In \S \ref{subsec:2gract}, we establish the existence of a canonical   \emph{Lie 2-group action} on $\Ho(\CF)$ whose orbits lie inside the $\Xi$-fibers, but which might fail to be the whole fiber (see Prop. \ref{prop:Lie.2.grpACTION} and Cor. \ref{prop:Lie.2.grpFIBER}).

\subsection{The lifted group action}\label{sec:groidorbits}

We show that the $G$ action on $P$ admits a canonical lift to $\Ho(\CF)$. We start with the following lemma.

\begin{lem}\label{def:gW}
Let $\widehat{g}\colon P\fto P$ be the diffeomorphism given by the action of $g\in G$. Take a path holonomy atlas $\CU$ and a bisubmersion $W\in \CU$. The  triple $$gW:=(W,\bt_g:=\widehat{g}\circ\bt, \bs_g:=\widehat{g}\circ\bs)$$ is a bisubmersion. Moreover $gW$ is adapted to $\CU$.
\end{lem}  
\begin{proof}
{Because the $G$-action preserves $\cF$, the pullback foliation $\widehat{g}^{-1}\CF$ equals $\CF$, implying that $gW$ is a bisubmersion.}

{We prove that $gW$ is adapted to $\CU$. Notice that   $g(W_1\circ W_2)=gW_1\circ gW_2$ for any $W_1,W_2\in \CU$. Hence it is sufficient to assume that $W$ is a path holonomy bisubmersion, as any element in the path-holonomy atlas is a composition of such elements.

Denote by $v_1,\cdots,v_n\in \CF$ the vector fields that give rise to the path holonomy bisubmersion $W$ (hence $W\subset \RR^n\times P$). Consider the push-forward vector fields $\widehat{g}_*v_1,\cdots,\widehat{g}_*v_n \in \CF$. The associated path-holonomy bisubmersion is defined on
 $$W':=\{(v,gp)\st (v,p)\in W\}\subset \RR^n\times P.$$
Since $gW\fto W', (v,p)\mapsto (v,gp)$ is an isomorphism of bisubmersions and since $W'$ is adapted to $\CU$ (being a path holonomy bisubmersion), we conclude that $gW$ is adapted to $\CU$.}
\end{proof}

Now {we introduce} the {\bf lifted action} 

\begin{equation}\label{eq:vecstar}
  \vec{\star}\colon{G} \times \Ho(\cF)\fto \Ho(\cF), \hspace{1cm}
  g\vec{\star} [v]:=[v]_{gW}  
\end{equation}
where, for any $v$ in a path holonomy bisubmersion $W$, we denote by 
{$[v]_{gW}$}  the class of $v$ regarded as an element of $gW$.
 {This is clearly well-defined and indeed a Lie group action.}
{Further, this action is by groupoid automorphisms:}

{\begin{lem}\label{lem:cover}
	For all $g\in G$ the map $g\vec{\star} (-)\colon \Ho(\cF)\fto \Ho(\cF)$ is a groupoid morphism covering the diffeomorphism $\widehat{g}\colon P\fto P$. 
\end{lem}}
 
\begin{proof}
 Using the construction of $g\vec{\star}(-)$ it is clear that 
 the source and target map commute with the map $\widehat{g}$. {Further
 $g\vec{\star}(-)$ preserves the groupoid composition since
  $g(W_1\circ W_2)=gW_1\circ gW_2$ for any path-holonomy bisubmersions $W_1,W_2$.} 
\end{proof}

{\begin{lem}\label{lem:orbits}
The  orbits of the lifted action $\vec{\star}$ lie in the fibers of $\Xi:\Ho(\CF)\fto \Ho(\CF_M)$.
\end{lem}}
\begin{proof}{Let $\CU$ be a source connected atlas for $\CF$.} 
The morphism $\Xi$ is induced by the identity map $\CU$ to $\pi\CU$,  see  the characterization of $\Xi$ given in Rem. \ref{rem:S}.

 Fix $g\in G$, and $u\in U\in \CU$.
  {By the above and since $\pi \circ \widehat{g}=\pi$,} the images under $\Xi$ of both $[u]$ and $g\vec{\star} [u]=[u]_{gU}$
are the class of the element $u\in (U, \pi\circ\bt, \pi\circ\bs)$. In particular,
$\Xi([u])=\Xi(g\vec{\star} [u])$, showing the desired statement.
\end{proof}

 \begin{ex}\label{ex:ri.regfol}({\bf Regular foliations}) {As discussed in Rem. \ref{rem:hol.regfol}, if $\CF$ is a regular foliation then $\CH(\CF)$ coincides with the classical notion of  holonomy groupoid given by holonomy classes of paths in the leaves. The Lie group $G$ acts canonically on $\CH(\CF)$ 
{by translating paths}: $G\times \CH(\CF)\fto \CH(\CF); (g,[\gamma]_{hol})\mapsto [g\gamma]_{hol}$. It is easy to see that this action agrees with the lifted action 
\eqref{eq:vecstar}, i.e. that
$[g\gamma]_{hol}=g\vec{\star}[\gamma]_{hol}$, because $[g\gamma]_{hol}\in (\CH(\CF),\bt,\bs)$ is equivalent to $[\gamma]_{hol}\in (\CH(\CF),\bt_g, \bs_g)$.} 
{ Moreover, using the characterization of $\Xi$ given by Prop. \ref{prop:regXi} it is clear that $\Xi(g\vec{\star}[\gamma]_{hol})=\Xi([\gamma]_{hol})$, in accordance with Lemma \ref{lem:orbits}.}
 \end{ex}

\subsection{{A canonical Lie 2-group action on the holonomy groupoid}}\label{subsec:2gract}

In this subsection we prove that there always is a Lie 2-group action on $\Ho(\CF)$ whose orbits lie inside the fibers of the morphism $\Xi$. In general however the orbits do not coincide with the (connected components of) the $\Xi$-fibers.
The formulae for this Lie 2-group action are suggested by the special case we will spell out in \S\ref{subsec:alternative}.

Denote the Lie algebra of $G$ by $\g$, and by $v_x\in \vX(P)$ the generator of the action corresponding to $x\in \g$.

\begin{lem}\label{lem:ideal}
The subspace $\h:=\{x\in \g: v_x\in \widehat{\cF}\}$ is a Lie ideal of $\g$.
\end{lem}
\begin{proof}
Since $\cF$ is $G$-invariant, for all $y\in \g$ we have $[v_y,\cF]\subset \cF$, or equivalently $[v_y,\widehat{\cF}]\subset \widehat{\cF}$.
Let $x\in \h$. Then for all $y\in \g$ we have $v_{[y,x]}=[v_y,v_x]\in \widehat{\cF}$, that is,
$[y,x]\in \h$. 
\end{proof}

Denote by $H$ the unique connected Lie subgroup of $G$ with Lie algebra $\h$. 
Lemma \ref{lem:ideal} implies that $H$ is a normal subgroup, hence as in Example \ref{ex:GH} we obtain a Lie 2-group $H\rtimes G\soutar G$.

We  define {a Lie group action of $H$} of $\Ho(\CF)$. It is not by groupoid automorphisms, unlike the  lifted $G$-action $\vec{\star} $ introduced in \S \ref{sec:groidorbits}, {but rather it preserves every source fiber}. In order to do so, we need a lemma.

\begin{lem}\label{lem:phi}
{There is a canonical groupoid morphism   $$\phi\colon {H}\times P\to \Ho(\CF),$$ where ${H}\times P$ denotes the transformation groupoid of the ${H}$-action on $P$  {obtained restring the action of $G$.}

{The morphism $\phi$ can be described as follows:} take $({h},p)\in {H}\times P$ and denote by $\widehat{{h}}\colon P\fto P$ the diffeomorphism corresponding to ${h}$ under the $G$-action. Then
$\phi({h},p)$ is the unique element of $\Ho(\CF)$ carrying the diffeomorphism $\widehat{{h}}$ near $p$.} 
\end{lem}

\begin{proof}[Proof of Lemma \ref{lem:phi}]
{Denote by $\cF_{H}$ the regular foliation on $P$ by orbits of the $H$-action.
Its holonomy groupoid is exactly $H\times P$, as follows from \cite[Ex. 3.4(4)]{AndrSk} (use that the Lie groupoid $H\times P$ gives rise to the foliation $\cF_{H}$ and is effective, i.e. the identity diffeomorphism on $M$ is carried only by identity elements of the Lie groupoid,  
due to the freeness of the action).} 

Since $\cF_{H}\subset \cF$, we are done applying \cite[Lemma 4.4]{SingSub} in the special case of the pair groupoid over $P$. 

 The description of $\phi$ given in the statement holds since $\phi$ is a groupoid morphism covering $Id_P$.
\end{proof}

\begin{lem}\label{lem:Lie.2.grpFIBER}
{   We have $\phi(H\times P)\subset \CK:=\ker(\Xi)$. }
\end{lem}
\begin{proof}
We use Lemma \ref{lem:phi}.
A point $\phi(h,p)$ of the l.h.s.  carries near $p$ the diffeomorphism $\widehat{{h}}$ 
(the diffeomorphism corresponding to ${h}$ under the $G$-action). If $h\in H$ is sufficiently close to the unit element, $\phi(h,p)$ admits a representative $u$ in a path holonomy bisubmersion 
$(U,\bt, \bs)$ for $\cF$
satisfying the properties of Rem. \ref{rem:S}. The point $u$, viewed as a point in $(U,\pi \circ\bt, \pi\circ\bs)$, carries $Id_M$, {since $\widehat{{h}}$ preserves each $\pi$-fiber}. This implies that  $\Xi([u])=1_{\pi(q)}$, {by   the characterization of $\Xi$ given in Rem. \ref{rem:S}}.
\end{proof}

{\begin{ex}\label{ex:le.regfol}({\bf Regular foliations}) As discussed in Rem. \ref{rem:hol.regfol}, if $\CF$ is a regular foliation there is a groupoid morphism  $Q\colon \Pi(\CF)\fto \Ho(\CF)$. The orbits of the $H$-action lie inside the leaves of $\CF$, {hence} for every path $h(t)$ {in $H$} and $p\in P$ the homotopy class $[h(t)p]$ {is an element of} $\Pi(\CF)$. Moreover, the freeness of the $G$-action implies that the elements $[h(t)p],[\widetilde{h}(t)p]\in \Pi(\CF)$ have the same holonomy if and only if $h(1)=\widetilde{h}(1)$ and $h(0)=\widetilde{h}(0)$. Therefore there is a well-defined injective groupoid morphism 
$$H\times P\fto \Ho(\CF);\;\; (h,p)\mapsto Q[h(t)p],$$
where $h(t)$ is any path in $H$ with $h(0)=e$ and $h(1)=h$. This morphism is precisely $\phi$. It is clear using Prop. \ref{prop:regXi} that its image lies inside $\CK=\ker(\Xi)$.
\end{ex}}

{Consider now the following {map}, obtained applying the morphism $\phi$ of Lemma \ref{lem:phi} and left-multiplying:}

\begin{equation}\label{eq:left}
   \cev{\star}\colon {H} \times \Ho(\cF)\fto \Ho(\cF), \hspace{1cm}
  h\cev{\star} {\xi}:={\phi(h,\bt({\xi}))\circ {\xi}}.
\end{equation}
Notice that $\phi$ being a groupoid morphism implies that $\cev{\star}$ is group action. {We now assemble the group action $\cev{\star}$ and the lifted action $\vec{\star}$:}

\begin{prop}\label{prop:Lie.2.grpACTION}
The map $$\star\colon (H\rtimes G)\times \Ho(\cF)\fto \Ho(\cF), \;\;(h,g)\star \xi:=h\cev{\star} \left(g\vec{\star} \xi\right)$$
is a Lie 2-group action.
\end{prop}
\begin{proof}
We first observe that if $\xi\in \Ho(\cF)$ carries a diffeomorphism $\psi$, then $g\,\vec{\star}\,\xi$ carries the diffeomorphism  $\widehat{g}\psi\widehat{g}^{-1}$.

We also observe the following two facts, which hold because both the left and the right side carry  the same diffeomorphism  (as can be seen using the above observation) and because of  Prop. \ref{prop:equiv2}:

i) the map $\phi\colon H\times P\to \Ho(\CF)$ satisfies the following equivariance property: $$g\,\vec{\star} \,\phi(h,p)=\phi(c_gh,gp),$$
 where $c_g$ denotes conjugation by $g$.

ii) For all $h\in H$ and $\xi \in\Ho(\CF)$ we have $$h\,\vec{\star}\,\xi=\phi(h,\bt(\xi))\circ \xi \circ \phi(h^{-1},h\bs(\xi)).$$

To show that $\star$ is a group action, the main requirement   is to show that
$((h_1,g_1)(h_2,g_2))\star\xi$ equals $(h_1,g_1)\star \big((h_2,g_2)\star\xi\big)$ for all $(h_i,g_i)\in H\rtimes G$ and $\xi \in\Ho(\CF)$.
This holds by a straightforward computation, in which the second term is re-written using the fact that $\phi$ is a groupoid morphism and is $G$-equivariant (fact i) above).

To show that $\star$ is a groupoid morphism,  the main requirement is to show that $(h_1h_2,g_2)\star(\xi_1\circ \xi_2)$ and $\big((h_1,g_1)\star\xi_1\big)\circ \big((h_2,g_2)\star\xi_2\big)$ agree, where $g_1=h_2g_2$ and $\bs(\xi_1)=\bt(\xi_2)$. Upon using that  $\phi$ is a groupoid morphism and the action $\vec{\star}$ is by groupoid automorphisms, this boils down to applying\footnote{with $h:=h_2$ and $\xi:=g_2\vec{\star}\xi_1$.} fact ii) above.
\end{proof}

{
\begin{ex}\label{ex:le.regfol.tot}({\bf Regular foliations})
When both   $\CF$ and $\CF_M$ are regular foliations, using Ex. \ref{ex:ri.regfol}
and Ex. \ref{ex:le.regfol}, 
 the Lie 2-group action $\star$
can be described as follows: for any path $\gamma$ in a leaf of $\CF$,
$$(h,g)\star [\gamma]=[h(t)\cdot g\gamma(1)]\circ [g\gamma],$$
where $t\mapsto h(t)$ is any path in the Lie group $H$ starting at the unit element and ending in $h$, and the dot denotes the group action of $H$ on $P$. Thus the right hand side is the (holonomy class of the) concatenation of the following two paths in $P$:
the translate $g\gamma$ of the original curve $\gamma$ by the group element $g$, and the path obtained acting on its endpoint $g\gamma(1)$ by the path $h(t)$ in $H$.  \end{ex}
}

\begin{figure}[H]\label{fig:comp}
	\centering
	\scalebox{.50}{\includegraphics{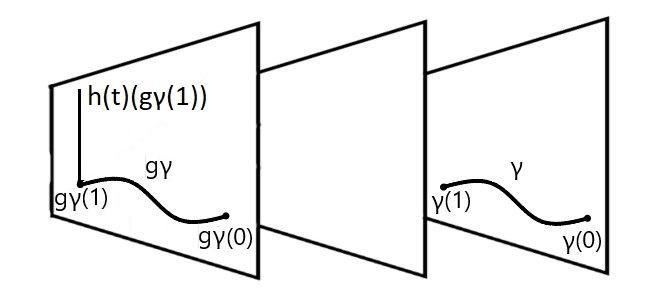}}
	\caption{{On the left, a curve representing the holonomy class $(h,g)\star [\gamma]$}.}
\end{figure}

{Lemma   \ref{lem:orbits} and Lemma \ref{lem:Lie.2.grpFIBER}  imply:}
\begin{cor}\label{prop:Lie.2.grpFIBER}
The orbits of the Lie 2-group action $\star$ of Prop. \ref{prop:Lie.2.grpACTION} lie inside the $\Xi$-fibers.
\end{cor}

\subsection{{An alternative description for the Lie 2-group action of \S \ref{subsec:ex.lie.2.grp}}}\label{subsec:alternative}

 We obtained the formula for the Lie 2-group action of $H\rtimes G$ in \S\ref{subsec:2gract} by considering a special case, as we now explain. Assume the set-up of Thm. \ref{thm:act.q.fol}, in particular that  $\Gamma_c(\ker d\pi)\subset \CF$ (i.e., $\widehat{\CF}$ contains the infinitesimal generators of the  $G$-action). There  we  defined an action  of $G\rtimes G$ on $\Ho(\CF)$ by means of the canonical isomorphism 
 $\varphi\colon\Ho(\CF)\xrightarrow{\sim}\pi^{-1}\Ho(\CF_M)$. 
 
The goal of this subsection is to prove the following proposition:

\begin{prop}\label{prop:sameaction}
{Consider these Lie 2-group actions of $G\rtimes G$ on $\Ho(\CF)$:
	\begin{itemize}
\item the action\footnote{Note that, under our assumptions on $\CF$, in the setting of Prop. \ref{prop:Lie.2.grpACTION} we have $H=G$, because the Lie subalgebra $\h$ introduced in Lemma \ref{lem:ideal} equals  the whole of $\g$.}
 $\star$ described in Prop. \ref{prop:Lie.2.grpACTION},
\item  the action described in Thm. \ref{thm:act.q.fol}. 
\end{itemize}
These two actions coincide. In other words, under the canonical isomorphism $\varphi\colon \Ho(\CF)\xrightarrow{\sim} \pi^{-1}\Ho(\CF_{M})$ given in Thm. \ref{thm:Homeo} we have}
 $$\varphi((h,g)\star \xi)=(h,g)* \varphi(\xi)$$ 
for all $(g,h)\in G\rtimes G\y \xi\in\Ho(\CF)$, where  $*$ is the action given in eq. \eqref{eq:star}.	
\end{prop}

 We will prove this statement by analyzing the restrictions of the actions to 
$\{e\}\times G$ and $G\times\{e\}$. (Notice that these two subgroups generate $G\rtimes G$ as a group, since every element $(h,g)$ can be written as $(h,e)(e,g)$).
{We denote by ${\varphi}\colon \Ho(\CF)\xrightarrow{\sim} \pi^{-1}\Ho(\CF_{M})$  the canonical isomorphism  given in Thm. \ref{thm:Homeo}.}
\begin{lem}\label{lem:raction}
Under the  isomorphism ${\varphi}$,  
the lifted action $\vec{\star} $ of $G$ introduced in \S\ref{sec:groidorbits} 
and the restriction of the Lie $2$-group action $*$ of {eq. \eqref{eq:star}}
agree:  $${\varphi}(g\vec{\star} \xi)=(e,g)* {\varphi}(\xi)$$
for all $g\in G$ and $\xi\in \Ho(\CF)$.
\end{lem}

\begin{proof}
Let $\CU$ be a path holonomy atlas for $\CF_M$. Recall that by Rem. \ref{rem:varphi}, the pullback-atlas $\pi^{-1}\CU$ is an atlas for $\pi^{-1}(\CF_M)$  equivalent to a path-holonomy atlas.  
Fix $\xi\in \Ho(\CF)$. Take a representative $w$ in a {bisubmersion} $W$ in the
path holonomy atlas of $\pi^{-1}(\CF_M)$.  
By the above, there is a path holonomy bisubmersion $U$ in $\CU$ and a locally defined morphism of bisubmersions $$\tau\colon (W,\bt,\bs)\fto   (\pi^{-1} U,\bt,\bs)$$ 
	mapping $w$ to some point $(p,v,q)$.
 By definition,  $\varphi([w])=(p,[v],q)$.
 
Now fix $g\in G$. Recall that the bisubmersion $gW:=(W, \widehat{g}\circ\bs, \widehat{g}\circ\bt)$ was defined in Lemma \ref{def:gW}.	
The same map $\tau$ is also a morphism of bisubmersions $$gW \fto   (\pi^{-1} U,\widehat{g}\circ\bt,\widehat{g}\circ\bs).$$ 
The latter bisubmersion is isomorphic to  $(\pi^{-1} U, \bt,\bs)$ via
$(p',v',q')\mapsto (gp',v',gq')$. By composition we obtain a morphism of bisubmersions
$gW \to (\pi^{-1} U, \bt,\bs)$ mapping $w$ to $(gp,v,gq)$. Hence $g\vec{\star} [w]:=[w]_{gW}$ agrees with $[(gp,v,gq)]\in \Ho(\CF)$, and therefore under $\varphi$ it is mapped to 
$(gp,[v],gq)=(e,g)* \varphi([w])$.
\end{proof}

\begin{lem}\label{lem:laction}
{Under the  isomorphism ${\varphi}$,  
the  action $\cev{\star} $ of $H$ introduced in eq. \eqref{eq:left} 
and the restriction of the Lie $2$-group action $*$ of {eq. \eqref{eq:star}}
agree:} $$\varphi(h\cev{\star} \xi)=(h,e)* \varphi(\xi)$$  for all $h\in G$ and $\xi\in \Ho(\CF)$.
\end{lem}
  
\begin{proof} Take {an arbitrary element $\xi\in \Ho(\CF)$ and} a path holonomy atlas $\CU$ for $\CF_M$. Let $(p,u,q)\in \pi^{-1}U\in \pi^{-1}\CU$ { be a representative of $\xi$}, and let $f$ be a local diffeomorphism carried at $(p,u,q)$.

Fix $h\in G$, and  denote by {$\widehat{h}\colon P\fto P$} the diffeomorphism {corresponding to $h$ under the $G$-action}. Note that {the transformation groupoid $G\times P$ carries}  the diffeomorphism $\widehat{h}$ at $(h,p)$. Hence any representative  of ${\phi(h,p)}\in \Ho(\CF)$ in $\pi^{-1}\CU$ carries this diffeomorphism, {where $\phi$ is the groupoid morphism of Lemma \ref{lem:phi}}.
In turn, this implies that any representative   of $\phi(h,p)\circ[(p,u,q)]\in \Ho(\CF)$ will carry $\widehat{h}\circ f$. Note that $(hp,u,q)\in \pi^{-1}U$ also carries $\widehat{h}\circ f$. {By the definition of holonomy groupoid (using Prop. \ref{prop:equiv2}) it follows that} $h\cev{\star} [(p,u,q)]=[(hp,u,q)]$.

We conclude that
$$\varphi(h\cev{\star} [(p,u,q)])=\varphi([(hp,u,q)])=(hp,[u],q)=(h,e)* \varphi([(p,u,q)]),$$
{using in the second equality the description of the isomorphism $\varphi$ given in Rem. \ref{rem:varphi}.}
\end{proof}

\begin{proof}[Proof of Prop. \ref{prop:sameaction}] The proposition follows from
	$$\varphi((h,g)\star \xi) = \varphi(h\cev{\star}\left(g\vec{\star} \xi\right))=(h,e)*((e,g)*\varphi( \xi)) = (h,g)*\varphi( \xi),$$
	{where we used Lemmas \ref{lem:raction} and \ref{lem:laction} in the second equality.}
\end{proof}
 
\begin{rem}\label{rem:imker}
{The image of $\phi\colon G\times P\fto \Ho(\CF)$ is the kernel of $\Xi\colon \Ho(\CF)\fto \Ho(\CF_M)$.} 
{Indeed, for every $(h,q)\in G\times P$, we have $\phi(h,q)=h\cev{\star} 1_q\in \Ho(\CF)$. Under the identification  $\varphi\colon \Ho(\CF)\xrightarrow{\sim} \pi^{-1}\Ho(\CF_{M})$,   this element corresponds to $(h,e)* \varphi(1_q)=(hq,1_{\pi(q)},q)\in \pi^{-1}\Ho(\CF_{M})$ by Lemma \ref{lem:laction}. Under the same identification, $\ker(\Xi)$ corresponds to $\ker(pr_2)=P\times_{\pi} M \times_{\pi} P$, by Prop. \ref{prop:normalq}.}
\end{rem}

\section{{Groupoid fibrations}}\label{section:groidfib}
 
In this section we investigate when the map $\Xi$ of Thm. \ref{thm:sur.hol} is a fibration. Loosely speaking,  fibrations are the notion of ``nice'' quotient map in the category of (topological or Lie) groupoids.

\subsection{Background on fibrations}

{We first present the recent notion of fibration for open topological groupoids.}

\begin{defi}\label{def:fibrtop}{\cite[Def. 2.1]{BussMeyer}}
Let $\CH\soutar P$ and $\Ho'\soutar M$ be  open topological groupoids. A   morphism of topological groupoids $\Xi\colon \CH\fto \Ho'$ covering a continuous map $\pi\colon P\fto M$ is called  \textbf{fibration} if 
\begin{equation}\label{eq:ximap}
{\Xi_\bs}\colon \CH\fto \CH' {}_\bs\!\times_\pi P ;\;\;\xi\mapsto (\Xi(\xi), \bs(\xi))
\end{equation}
is a surjective open map.
  \end{defi}

\begin{rem}
{In Def. \ref{def:fibrtop} the base map $\pi\colon P\fto M$ {is required to be neither open nor surjective}; see \cite[Remark 2.7]{BussMeyer} for a motivation of this choice. In 
all the instances of fibration appearing in this note, the base map is a surjective submersion between manifolds, thus also open.}
\end{rem}

{We now review fibrations of Lie groupoids and equivalent characterizations, after   \cite{MK2}.}

\begin{defi}\label{def:fibrLie} {\cite[Def. 2.4.3]{MK2}} Let $\CH\soutar P$ and $\Ho'\soutar M$ be Lie groupoids. A   morphism of Lie groupoids $\Xi\colon \CH\fto \Ho'$ covering a smooth map $\pi\colon P\fto M$ is called  \textbf{fibration} if    $\pi$ and  ${\Xi_\bs\colon }\CH\fto \CH' {}_\bs\!\times_\pi P$ (as in {Eq.} \eqref{eq:ximap}) are both surjective submersions.
\end{defi}
{The surjectivity conditions in the previous definition assure} that for any two composable elements in $\CH'$ there exist composable preimages in $\CH$, and thus
the composition on $\CH'$ is entirely determined by the composition in $\CH$. 

In \cite{MK2}  two notions are used to describe  {fibrations of} Lie groupoids: smooth congruences and normal subgroupoid systems.

 Before introducing them, recall that a \textbf{smooth equivalence relation} on a manifold $P$ is an embedded wide Lie subgroupoid $R$ of the pair groupoid $P\times P$.
By the Godement criterium, $P/R$ is a smooth manifold such that the projection map from $P$ is a submersion if and only if $R$ is a smooth equivalence relation. 

\begin{defi}\label{def:smt.cong}{\cite[2.4.5]{MK2}} Let $\CH\soutar P$ be a Lie groupoid. A \textbf{smooth congruence} on $\CH$ consists of two smooth equivalence relations $\CR$ on $\CH$ and $R$ on $P$, such that:
	\begin{itemize}
		\item $\CR\soutar R$ is a Lie subgroupoid of the Cartesian product $\CH\times \CH \soutar P\times P$,
		
		\item the map $\CR\fto \CH\; {}_\bs \!\times_{\mathrm{Pr}_1} R;\;\;(\xi_2,\xi_1)\mapsto (\xi_2, \bs(\xi_2),\bs(\xi_1))$ is a surjective submersion.\footnote{In \cite{MK2} Mackenzie  describes this condition as a certain square diagram being ``versal''.}
	\end{itemize}	
\end{defi}

{Normal subgroupoid systems allow to describe a Lie groupoid fibration  in terms of data on the domain, analogously to how a surjective group morphism can be described by its kernel.} Note that, if $\CK$ is a closed embedded wide Lie subgroupoid of $\CH$, then, by the Godement criterion, the set {of left cosets}
$$ \CK\backslash \CH=\{ \CK\circ \xi \st \xi\in \CH\}$$
has an unique manifold structure making the quotient map $\CH \fto \CK\backslash \CH; \;\xi\mapsto \CK \xi{:=(\CK\circ \xi)}$ a surjective submersion. Note also that the source map $\bs\colon \CH\fto P$ quotients to a well-defined surjective submersion $\CK\backslash \CH\fto P$, which we also denote by $\bs$.

\begin{defi}\label{def:nss}{\cite[2.4.7]{MK2}}
	A \textbf{normal subgroupoid system} in $\CH\soutar P$ is a triple $(\CK,R,\theta)$ where $\CK$ is a closed, embedded, wide Lie subgroupoid of $\CH$; $R$ is a smooth equivalence {relation} on $P$; and $\theta$ is an action of $R$ on the map $\bs\colon \CK\backslash \CH\fto P$ such that for all $(p,q)\in R$ the following   holds:
	\begin{enumerate}
		\item Let $\xi\in \CH$ with $\bs(\xi)=q$ and $\xi_1\in \CH$ with $\theta(p,q)(\CK \xi)=\CK \xi_1$, then $(\bt(\xi_1),\bt(\xi))\in R$.
		\item $\theta(p,q)(\CK {e_q})=\CK {e_p}$.
		\item Let $\xi_1$ and $\xi_2$ composable elements in $\CH$ such that $\bs(\xi_2)=q$. Consider  $\xi_1'$ and $\xi_2'$ such that $\theta(p,q)(\CK \xi_2)=\CK \xi_2'$ and $\theta(\bt(\xi_2'),\bt(\xi_2))(\CK \xi_1)=\CK \xi_1'$, then:
		$$\theta(p,q)(\CK (\xi_1\circ \xi_2))=\CK (\xi_1'\circ\xi_2').$$
	\end{enumerate}
\end{defi}

The following theorem  says that smooth congruences, fibrations and normal subgroupoids systems are equivalent descriptions for the quotients of Lie groupoids.

\begin{thm}\label{thm:norm.sub.sys}(\cite[2.4.6]{MK2} and \cite[2.4.8]{MK2})
{ Let $\CH\soutar P$ be a Lie groupoid.}
\begin{enumerate}
	\item If $\Xi$ is a fibration {defined on $\CH$} covering $\pi$, then the pair $(\CR,R)$ is a congruence, where $\CR:=\CH\times_\Xi \CH$ and $R:= P\times_\pi P$. Conversely, given a congruence, the {induced} quotient map is a fibration.
	\item If $(\CK,R,\theta)$ is a normal subgroupoid system {on $\CH$} then $R$, together with the equivalence relation on $\CH$ given by
	$$\CR:=\{(\xi',\xi)\in \CH\times\CH \st (\bs(\xi'),\bs(\xi))\in R \y \theta(\bs(\xi'),\bs(\xi))\CK \xi=\CK \xi'\},$$
{is} a smooth congruence on $\CH$.
	
	 Conversely, let $(\CR,R)$ be a smooth congruence on $\CH$. Let $\CK$  
{consist of elements of $\CH$ which are related to an identity element
 $e_p$, where $p\in P$.} 	 

Let $\theta$ be the following   action of $R$ on $\CK\backslash \CH$: for every $(p,q)\in R$ and $\xi\in \CH$ with source $q$,
	$$\theta(p,q)(\CK \xi)=\CK \xi',$$
	where $\xi'$ is any element related to $\xi$ and with source $p$. 	
Then  $(\CK,R,\theta)$ is a normal subgroupoid system {on $\CH$}.
\end{enumerate}
\end{thm}

\subsection{{The morphism $\Xi$ is not always a fibration}}
 {Let $\pi\colon P\fto M$ be a surjective submersion with connected fibers,   let $\CF$ be a singular foliation on $P$ satisfying the condition \eqref{eq:bracketpres}, and
denote by $\CF_M$ the induced singular foliation on $M$.}
{In Theorem \ref{thm:sur.hol} we obtained an open surjective morphism of topological groupoids
$\Xi \colon \Ho(\CF)\fto \Ho(\CF_M)$.} 

{This morphism is not always a fibration, as the following example shows. What fails here is the surjectivity of the map {$\Xi_\bs$} ({as in Eq.} \eqref{eq:ximap}).}

\begin{ex}\label{ex:fail} Let $P=\RR^2-(\{0\}\times \RR_+)$ {(the plane with a vertical half-line removed)}, $M=\RR$ and $\pi\colon P\fto M$ the first projection. Take $\CF$ to be the foliation on $P$ given by the horizontal lines {and halflines}, then $\CF_M$ is the full foliation on $M$.
{In particular, the foliation $\CF$ has no holonomy, and $\Ho(\CF_M)=M\times M$.} 
	
The map ${\Xi_\bs\colon} \Ho(\CF)\fto \Ho(\CF_M) {}_\bs\!\times_\pi P$  (as in \eqref{eq:ximap})  is not surjective. Indeed, take ${\zeta}\in \CH(\CF_M)$ such that $\bs({\zeta})=1$ and $\bt({\zeta})=-1$. {For all $y>0$,} the element $({\zeta},{(1,y)})$ lies in $\Ho(\CF_M)_\bs\!\times_\pi P$, but there is no ${\xi}$ in $\CH(\CF)$ satisfying $\bs({\xi})={(1,y)}$ and $\Xi({\xi})={\zeta}$     (otherwise   $\bt(\xi)$ would be of the form $(-1,*)$ and thus not lie in the same leaf as $\bs({\xi})$, leading to a contradiction). 
 \end{ex}

\begin{rem}
{ The above example also shows that, in general, Lie algebroid fibrations 
do not integrate to Lie groupoid fibrations.
To explain this, recall that a Lie algebroid fibration  \cite{MK2} is a morphism of  Lie algebroids $\phi\colon A\to B$ which is fiberwise surjective and covers a surjective submersion between the manifolds of objects. Suppose $A,B$ integrate to source simply connected Lie groupoids  $\cG, \cH$. Then the unique Lie groupoid morphism $\Phi\colon \cG \to \cH$ integrating $\phi$ is a surjective submersion, but
in general   fails to be a  Lie groupoid fibration. An instance is   the above example, in which $A$ and $B$ are the involutive distributions tangent to the foliations, $\phi=\pi_*$, and consequently $\Phi=\Xi$.
}

It was pointed out to us that a sufficient condition for a Lie algebroid fibration 
to integrate to Lie groupoid fibration is the existence of a complete Ehresmann connection, as realized by Brahic in \cite{BrahicExt}. See Example 4.12 there  for an instance involving foliations.\end{rem}

\subsection{Fibrations from pullback-foliations}
{In the next two subsections we present two cases in which $\Xi$ is a fibration. A simple instance is when $\CF$ is a pullback-foliation:}

\begin{prop}\label{prop:pullbackfibr}
	As in Prop. \ref{prop:submfol}, let $\pi\colon P\fto M$ be a surjective submersion with connected fibers, and let $\CF$ be a singular foliation on $P$ satisfying $\Gamma_c(\ker d\pi)\subset \CF$. 
	
	Then the map $\Xi\colon \CH(\CF)\fto \CH(\CF_M)$ of Thm. \ref{thm:sur.hol} is a fibration of topological groupoids.
\end{prop}

\begin{proof}   There is  an {isomorphism} of topological groupoids     $\varphi\colon \CH(\CF)\fto \pi^{-1}(\CH(\CF_M))$ by {Thm. \ref{thm:Homeo}. By Prop. \ref{prop:normalq}, the map $\Xi$ corresponds to}   the projection $\text{pr}_2\colon \pi^{-1}(\CH(\CF_M))\fto \CH(\CF_M)$, which is clearly a fibration.
\end{proof}

{
\begin{rem}
As the proof of Prop. \ref{prop:pullbackfibr} shows, when $\CH(\CF)$ and $\CH(\CF_M)$ are smooth, then $\Xi$ is a fibration of Lie groupoids. The normal subgroupoid system $(\CK,R,\theta)$ corresponding to it (as in Thm.  \ref{thm:norm.sub.sys}) is given by $\CK=P\times_M 1_M\times_M P$, by $R=
P\times_M P$, and the following  Lie groupoid action $\theta$ of $R$ on $\CK\backslash \CG\cong 
\CH(\CF_M)\times_M P$:
$$\theta(p,q)[(\xi,q)]=[(\xi,p)].$$ Here the square bracket denotes the equivalence class in $\CK\backslash \CG$.
\end{rem}
}

\subsection{Fibrations from group actions}

 We exhibit one more case when $\Xi$ is a fibration. {In this whole subsection we assume the set-up at the beginning of \S \ref{sec:general}, that is:
				a foliated manifold $(P,\cF)$, and
				a free and proper action of a connected Lie group $G$ on $P$ preserving $\cF$. }
				
		{		Denote $M:=P/G$, denote by $\CF_M$ the induced singular foliation there, and use the short-hand notation $\Ho:=\Ho(\CF)$ and $\Ho':=\Ho(\CF_M)$.
The groupoid morphism $\Xi$   defined in eq. \eqref{eq:Xi} and  the lifted Lie group action
  $\vec{\star}$ in eq. \eqref{eq:vecstar}, thanks to Lemma \ref{lem:cover} and Lemma \ref{lem:orbits}, provide us with the following data:} 
 
\begin{enumerate}
	\item a free and proper action of a Lie group $G$ on a manifold $P$, with quotient map $\pi\colon P\fto M:=P/G$,
	\item an {open} surjective morphism of  {holonomy groupoids}  $\Xi\colon \Ho\fto \Ho'$ covering $\pi$,
	\item a group action $\vec{\star}$ of  $G$ on $\Ho$ by groupoid automorphisms covering the $G$-action on $P$ and 
	preserving each fiber of $\Xi$.
\end{enumerate}

\[\begin{tikzcd}
\Ho \ar[d, shift right=.2em, swap,]\ar[d,shift left=.2em,] \ar[r,"\Xi"] & \Ho' \ar[d, shift right=.2em, swap]\ar[d,shift left=.2em] \\
P \ar[r,"\pi"] & M
\end{tikzcd}\]

\subsubsection*{Fibrations of topological groupoids}
 
\begin{prop}\label{prop:123fibr}
	Assume {the set-up at the beginning of \S \ref{sec:general}}.  {Make the following technical assumption: the diffeomorphism $G\times P\to G\times P, (g,p)\mapsto (g,\hat{g}(p))$ is an automorphism of the product foliation $0\times \cF$. (This is a reasonable assumption since, for every $g\in G$, $\hat{g}$ is an automorphism of $(P,\cF)$).} Then the map $\Xi$ is a  fibration of topological groupoids.
\end{prop}

\begin{proof} According to Def. \ref{def:fibrLie}, we need to show that the map
$$\Xi_{{\bs}}\colon\CH\fto \CH' {}_\bs\!\times_\pi P ;\xi\mapsto (\Xi(\xi), \bs(\xi))$$ is surjective {and open}. 
{For the surjectivity we argue as follows.}	Let $({\zeta},p)\in \CH' {}_\bs\!\times_\pi P$. There exists a $\xi\in \CH$ such that $\Xi(\xi)={\zeta}$. Then $\pi(\bs(\xi))=\pi(p)$,  which means that there exists a (unique) $g\in G$ such that $g\bs(\xi)=p$. This implies $\Xi_{\bs}(g\vec{\star}\xi)=({\zeta},p)$ i.e. $\Xi_{\bs}$ is surjective.

To prove that $\Xi_{\bs}$ is an open map, let $\cO\subset \Ho$ be an open subset. It suffices to prove that for any point $\xi\in \cO$ there is a subset $V\subset\cO$ containing $\xi$ such that $\Xi_{\bs}(V)$ is open. Take a slice  through $\bs(\xi)$, i.e. a
sufficiently small submanifold $S\subset P$ through $\bs(\xi)$ transverse to the $\pi$-fibers. {We will  construct such $V$ as $\widetilde{G}\vec{\star}\sigma$, where $\widetilde{G}$  is a  neighborhood   of the identity element  in $G$, 
and $\sigma$ is a neighborhood of $\xi$ in $\bs^{-1}(S)$.
}

\noindent{\bf Step 1:} \emph{There is a relatively compact neighborhood $\sigma$ of $\xi$ in $\bs^{-1}(S)$ such that its closure satisfies $\bar{\sigma} \subset \cO$.}

To construct $\sigma$, take a bisubmersion $U$ and a point $u\in U$ such that the quotient map $q_U\colon U\to H$ satisfies $q_U(u)=\xi$. As $U$ is a manifold, there exists a relatively compact (open) neighborhood $U_0$ of $u$, and furthermore we can assume that its closure $\bar{U_0}$ lies in the open subset $q_U^{-1}(\cO)$. 
Since $q_U$ is an open map by \cite[Lemma 3.1]{MEsingfol}, $q(U_0)$ is open. 
Hence $$\sigma:=q_U(U_0)\cap \bs^{-1}(S)$$ is an open subset of $\bs^{-1}(S)$.

We now show that $\bar{\sigma}$, the closure of $\sigma$ in $\bs^{-1}(S)$, is compact.
To do so we make use of  $\breve{\sigma}$, the closure of $\sigma$ in $\bs^{-1}(\bar{S})$. 
We have $\breve{\sigma}\subset q_U(\bar{U_0})\cap \bs^{-1}(\bar{S})$. This shows that
$\breve{\sigma}$ is compact, being a closed subset of a compact set. Further it shows that, 
shrinking $U_0$ if necessary, we can arrange that
 $\breve{\sigma}$ is contained in $\bs^{-1}({S})$. This in turn implies that $\bar{\sigma}\subset \breve{\sigma}$: the r.h.s is a closed subset of $\bs^{-1}({S})$ containing $\sigma$, and by definition the l.h.s. is the smallest such subset. Hence $\bar{\sigma}$, being a closed subset of a compact set, is itself compact.

{
\noindent{\bf Step 2:} \emph{There is a neighborhood $\widetilde{G}\subset G$ of the identity element $e$ such that $\widetilde{G}\vec{\star}\sigma \subset \cO$.}} 

{The preimage of $\cO$ under
the action map $\vec{\star}\colon G\times \Ho\to\Ho$ is open, and by Step 1 it contains $\{e\}\times \bar{\sigma}$.  Hence it contains $\widetilde{G}\times \bar{\sigma}$
for some neighborhood $\widetilde{G}$ of $e$, using the compactness of $\bar{\sigma}$}.

{
\noindent{\bf Step 3:} \emph{$V:= \widetilde{G}\vec{\star}\sigma$ is an open subset of $\cH$.}
}

{
Notice that, for every $g\in \widetilde{G}$, 
\begin{equation}\label{eq:setgstar}
g\vec{\star}\sigma=\{q_{gU}(v):v\in 
\bs_U^{-1}(S)\cap U_0\},
\end{equation}
 {where $q_{gU}$ denotes the quotient map of the bisubmersion obtained as in Lemma \ref{def:gW}.}
 We want to describe this set in terms of the bisubmersion $U$.
This is possible because $q_U(U)$ is an open subset of $\CH$ containing $\sigma$, thus $g\vec{\star}\sigma$ will lie in $q_U(U)$ provided $g$ is close enough to the unit element $e\in G$. In the following claim we might need to shrink $\widetilde{G}$
to a smaller neighborhood of $e$.
}

\noindent{\bf Claim:}
{
  \emph{Let $g\in \widetilde{G}$. There is  
morphism of bisubmersions  $\phi^g\colon gU\to U$ (defined only on an open subset of $gU)$) which is a diffeomorphism onto its image and depends smoothly\footnote{Recall that the manifold underlying the bisubmersion $gU$ is $U$, and thus is independent of $g$.}
 on $g$.}
}

{
Assume first that $U$ is a path holonomy bisubmersion, associated to local generators $X_1,\dots,X_k$ of $\cF$ near $\bs(\xi)$. Fix $Y_i\in \Gamma(\ker(d\bs_U))$ such that $d\bt_U(Y_i)=X_i$ (this is possible by the proof of \cite[Prop. 2.10 b)]{AndrSk}).
We then have $d\bt_{gU}(Y_i)=\hat{g}_*X_i$. There exist smooth functions $c_{ij}^g$, defined on an open subset of $\bs_U(U)\cap \hat{g}(\bs_U(U))$, such that $
\hat{g}_*X_i=\sum_jc_{ij}^gX_j$. Further, these functions can be chosen to vary smoothly with $g$, by the technical assumption in the statement of the proposition and since $X_1,\dots,X_k$ -- viewed as vector fields on $G\times P$ -- locally generate the singular foliation $0\times \cF$. 
}

{
Consider the map 
\begin{equation}\label{eq:morpgUU}
\phi^g\colon gU\to U,   \exp_{(0,x)}(\sum \lambda_iY_i)\mapsto \exp_{(0,\hat{g}x)}(\sum \lambda_i (\bt_U^*c_{ij}^g)Y_j),
\end{equation}
which actually is defined only on a open subset of $gU$ (namely, when $x\in \bs_U(U)\cap \hat{g}^{-1}(\bs_U(U))$.)
It preserves source fibers, since the $Y_i$ are tangent to the source fibers of both bisubmersions. Applying $\bt_{gU}$ to the above point of the domain and applying $\bt_U$ to its image we obtain the same point of $P$ (namely $\exp_{(0,\hat{g}x)}(\sum \lambda_i\hat{g}_*X_i)$). Thus $\phi^g$ is a morphism of bisubmersions. Further $\phi^g$ is a diffeomorphism onto its image
(see the proof of \cite[Prop. 2.10 b)]{AndrSk}; this was also checked explicitly in \cite[Lemma 2.6]{AZ1}). Finally, looking at \eqref{eq:morpgUU} it is clear that $\phi^g$ depends smoothly on $g$. 
}

In the general case, $U=U_l\circ\dots\circ U_1$ is a composition of path-holonomy bisubmersions. We can apply the above construction to obtain, for each $k=1,\dots,l$, a morphism of bisubmersions  $\phi_k\colon gU_k\to U_k$.
Assembling them, we obtain a 
morphism of bisubmersions  $\phi^g\colon gU\to U$ satisfying the properties of the claim. This proves the claim.

{
From eq. \eqref{eq:setgstar} and the claim we deduce that 
 $g\vec{\star}\sigma=\{q_{U}(\phi^gv):v\in 
\bs_U^{-1}(S)\cap U_0\}$.
Define $$\Phi\colon \widetilde{G}\times (\bs_U^{-1}(S)\cap U_0)\to U,\;\; (g,v)\mapsto \phi^g(v).$$ Then $\tilde{G}\vec{\star}\sigma=q_U(image(\Phi))$.
As $q_U$ is an open map, it suffices to argue that $image(\Phi)$ is open in $U$.
}

{The map $\Phi$ fits in the  commutative diagram}\\
\[\begin{tikzcd}
\widetilde{G}\times (\bs_U^{-1}(S)\cap U_0) \ar[rr,"\Phi"] \ar[rd,"pr_1"] & & U \ar[dl, "\alpha"]\\
 & \widetilde{G} &
\end{tikzcd},\]
{where $pr_1$ is the first projection and $\alpha$ is mapping $u$ to the unique group element $g$ such that $\bs_U(u)\in \hat{g}S$. The  maps  $pr_1$ and $\alpha$  are submersions. Further, for every $g\in \widetilde{G}$, the map  $$\Phi|_{pr_1^{-1}(g)}\colon \bs_U^{-1}(S)\cap U_0\to \alpha^{-1}(g)=\bs_U^{-1}(\hat{g}S)$$ equals $\phi^g$, which is an diffeomorphism onto its image by the above claim.} {Hence $\Phi$ is a diffeomorphism onto its image, which consequently  is open in $U$ by dimension reasons.}

{
\noindent{\bf Step 4:} \emph{$\Xi_{\bs}(V)$ is open in $ \CH' {}_\bs\!\times_\pi P$.}
}

 {
As $\Xi\times \bs\colon \Ho\times \Ho\to \Ho'\times P$ is an open map,  and by Step 3, the image of 
$V\times V$ under this map is open. Hence its intersection with $ \CH' {}_\bs\!\times_\pi P$ is open there. This intersection is exactly 
$$A:=\{(\Xi(v_1),\bs(v_2): v_1, v_2\in V \text{ and } \pi(\bs(v_1))=\pi(\bs(v_2))\}.$$
We now show that  $A=\Xi_{\bs}(V)$. We just need to prove ``$\subset$'', since the other inclusion is obvious.
To this aim, take an arbitrary element $a:=(\Xi(v_1),\bs(v_2))$ of $A$. Since $\bs(v_1)$ and $\bs(v_2)$ lie in the same $\pi$-fiber, there is a unique $g\in G$ such that $\hat{g}\cdot\bs(v_1)=\bs(v_2)$. It is immediate to check that  $\Xi_{\bs}(g \vec{\star} v_1)=a$, hence we are done if we show that $g \vec{\star} v_1\in V$. To this aim, write $v_i=g_i \vec{\star} \xi_i$ for unique 
$g_i\in \widetilde{G}$ and $\xi_i\in \sigma$ ($i=1,2$). We have $\bs(\xi_1)=\bs(\xi_2)$, since these elements belong to the same $\pi$-fiber and $\bs(\sigma)$ lies inside a slice transverse to such fibers. This implies that $g=g_2g_1^{-1}$. Hence $g \vec{\star} v_1=g_2 \vec{\star} \xi_1$, which by construction lies in $V$.
 }
\end{proof}

{Assuming the set-up at the beginning of \S \ref{sec:general}, in Lemma \ref{prop:fib.xi0} and Prop. \ref{cor:grpd.act.hol} we  describe   the fibers of $\Xi\colon \CH\fto \CH'$.} 
Denote $\CK:=\ker(\Xi)$, a  topological subgroupoid  of $\CH$ with space of objects $P$.  Given a fibration of open topological groupoids, the kernel -- called ``fibre'' in the terminology of \cite{BussMeyer} -- is an 
open topological subgroupoid   
\cite[Lemma 2.3]{BussMeyer}. Hence $\CK$ is an open topological subgroupoid  when the technical assumption in Prop. \ref{prop:123fibr} is satisfied.  

\begin{lem}\label{prop:fib.xi0}
	The fibers of $\Xi$ are given by the orbits of the action of $G$ on $\CH$ composed\footnote{Recall that the composition (multiplication) of the groupoid $\CH$ is denoted by $\circ$.}
	with elements in $\CK$. More precisely, the fiber through $\xi\in \CH$ is 
	$$\CK\circ (G\vec{\star} \xi) :=\{\chi\circ(g\vec{\star} \xi)  :\chi\in \CK \y g\in G\}.$$	
\end{lem}
\begin{proof}
	{Fix $\xi_1\in \CH$. {Thanks to Lemma \ref{lem:orbits},} the above subset $\CK\circ (G\vec{\star} \xi_1)$ is certainly contained in the $\Xi$-fiber through $\xi_1$.}
	
	{ To show the converse,} let $\xi_2$ lie in the same $\Xi$-fiber as $\xi_1$, then $\bs(\xi_2)$ and $\bs(\xi_1)$ lie in the same $\pi$-fiber. 
	Let $g\in G$ such that $\hat{g}\bs(\xi_1)=\bs(\xi_2)$. {As this equals $\bs(g\vec{\star} \xi_1)$, the groupoid composition} 
	{$\xi_2\circ (g\vec{\star} \xi_1)^{-1}$} is well-defined, and
	$$\Xi\left(\xi_2\circ (g\vec{\star} \xi_1)^{-1}\right)=\Xi(\xi_2)\circ\Xi(g\vec{\star} \xi_1)^{-1}=\Xi (\xi_2)\circ \Xi(\xi_1)^{-1}= 1_{\pi(\bt(\xi_2))},$$
	where we used that $\Xi$ is a groupoid morphism and the action of $G$ preserves the $\Xi$ fibers, respectively in the {first and second equality}. As a consequence, {$\xi_2\circ (g\vec{\star} \xi_1)^{-1}\in \ker \Xi=\CK$}.
\end{proof}

\begin{rem}
	While the fibers of a group {morphism} are just translates of the kernel, for morphisms of groupoids over different bases this is no longer true. This explains why the description of the fibers in Lemma \ref{prop:fib.xi0} is slightly involved.
\end{rem}

{A first consequence of Lemma \ref{prop:fib.xi0} is that} 
the fibers of $\Xi\colon \CH\fto \CH'$ {are orbits of a groupoid action}:

\begin{prop}\label{cor:grpd.act.hol} 
	There is a  {topological} groupoid structure on $\CK\times G$ and a  groupoid action of $\CK\times G$ on $\bt\colon \CH\fto P$, whose orbits coincide with the fibers of $\Xi$.	
\end{prop}
An instance of {this proposition} is Ex. \ref{ex:non.sing.q}, where we have $G=U(1)$ and $\ker\Xi =1_P$.

\begin{proof}
Since the Lie group $G$ acts by groupoid automorphisms on the groupoid $\CK$, we can form the semidirect product groupoid (see \cite[\S 2]{BrownCoefficients}  and \cite[\S 11.4]{TopGrpds}). We obtain a groupoid structure on $\CK\times G$ with space of objects $P$, for which
\begin{itemize}
	\item[a)] the source and target maps are respectively {$(\xi,g)\mapsto g^{-1}\bs(\xi)$ and $(\xi,g)\mapsto \bt(\xi)$},
 	\item[b)] the composition is $(\xi_2,g_2)\circ (\xi_1,g_1)\mapsto (\xi_2\circ (g_2\vec{\star} \xi_1),g_2 g_1)$.
\end{itemize}
One checks 
that the groupoid $\CK\times G$ acts on the map $\bt\colon \CH\fto P$ via
\begin{equation}\label{eq:stargroid}
\smallwhitestar\colon (\CK\times G) \times_P \CH {\to \CH};\;\; \left(( {\chi},g),{\xi}\right) \mapsto ( {\chi},g)\smallwhitestar {\xi}:={\chi}\circ(g\vec{\star} {\xi}).
\end{equation}
The orbits of this groupoid action 
 are precisely the fibers of $\Xi$, by Lemma \ref{prop:fib.xi0}. 
\end{proof}

\subsubsection*{Fibrations of Lie groupoids}				
	
{To conclude this section, we show that for Lie groupoids $\Xi$ is a fibration in the sense of Def. \ref{def:fibrLie}.}
Denote $\CK:=\ker(\Xi)$.  
{Denote $$R=
P\times_M P\cong G\times P,$$ where the isomorphism is given by $(q,p)\mapsto (g,p)$ for $g\in G$  the unique element satisfying $gp=q$.
Note that since the  action $\vec{\star}$ of $G$ on $\CH$ preserves each fiber of $\Xi$, we obtain by restriction a group action of $G$ on $\CK$, also by groupoid automorphisms. 
This ensures that the following groupoid action $\theta$ of $R$ on $\bs\colon \CK\backslash \CH\fto P$ is well-defined:
 $$\theta(g,p)(\CK {\xi})=\CK (g\vec{\star}{\xi}).$$
Hence $\theta$ essentially amounts to the lifted $G$ action.
}

\begin{prop}\label{prop:norsubG}  
 {	Assume the set-up at the beginning of \S \ref{sec:general}, and that  $\CH$ and $\CH'$ are Lie groupoids. 
	Then: 
	\begin{itemize}
	\item[i)] $\Xi$ is a   fibration of Lie groupoids,
	 	\item[ii)] 
	 $(\CK,R,\theta)$ is the normal subgroupoid system corresponding to it, via Thm. \ref{thm:norm.sub.sys}.
\end{itemize}
}	  
\end{prop}

\begin{proof}  
 We first prove that $(\CK,R,\theta)$ is a normal subgroupoid system. 
Because $\Xi$ and $\pi$ are surjective submersion {(see Rem. \ref{rem:ss})}, we have that $\CK$ is an closed, embedded wide Lie subgroupoid of $\CH$ and $R$ is a smooth equivalence {relation}.
	Because $\vec{\star}$ is a smooth Lie group action, 
	we have that $\theta$ is also a smooth action.
{The three conditions in Def. \ref{def:nss} are satisfied because $\vec{\star}$  is a group action on $\CH$ by Lie groupoid automorphisms, covering the group action of $G$ on $P$.}

Since $(\CK,R,\theta)$ is a normal subgroupoid system,   by Theorem \ref{thm:norm.sub.sys} $(\CR,R)$ is a smooth congruence, where
$$\CR=\{{(\xi',\xi)\in \CH\times\CH} \st (\bs(\xi'),\bs(\xi))\in R \y \theta(\bs(\xi'),\bs(\xi))\CK \xi=\CK \xi'\}.$$
{Lemma \ref{prop:fib.xi0} shows that} $(\xi',\xi)\in \CR$ if and only if $\Xi(\xi')=\Xi(\xi)$.  
This means that {the quotient map of the smooth congruence $(\CR,R)$ is exactly $
\Xi$. In particular, by Theorem \ref{thm:norm.sub.sys}, $\Xi$ is a fibration.}
\end{proof}

\begin{rem}\label{rem:hypo}
{Prop. \ref{prop:norsubG} holds also replacing the hypothesis that $\CH'$ is an embedded Lie groupoid with the hypothesis that $\CK$ is a Lie subgroupoid of $\CH$, with the same proof.}
\end{rem}

{When $\CH(\cF)$ is a Hausdorff Lie groupoid, Remark \ref{rem:S} allows to give a description of $\CK$ that does not  make   reference to the morphism $\Xi$. Namely, $\CK$ consists of the elements $\xi\in \CH(\cF)$ 
that carry a diffeomorphism $\varphi$ such that, 
for some slice $S$ through $\bs(\xi)$ transverse to the $\pi$-fibers,  this diagram commutes:}
\[\begin{tikzcd}
S  \ar[swap,"\pi",dr] \ar[rr,"\varphi"] & & \varphi(S)  \ar[dl,"\pi"] \\
& M  & 
\end{tikzcd}\]

{
In the following   simple example  we describe $\CH(\cF_M)$ using Prop. \ref{prop:norsubG}.}
\begin{ex}\label{ex:quotsim}
{Let $P=S^1\times \RR$ be the cylinder with coordinates $\theta$ and $y$, endowed with the (free) action of $G=(\RR,+)$ by vertical translations. The action preserves the (regular) foliation $$\cF:=\left\langle \pd{\theta} +\lambda \pd{y}\right\rangle$$ by spirals, where $\lambda$ is fixed non-zero real number. The quotient $M:=P/G$ is $S^1$. It is easy to see that the induced foliation $\CF_M$ is the full foliation, but we want to describe  
$\CH(\cF_M)$ without using this fact.
}

{
 We have $\CH(\cF)=\RR\times P$, the transformation groupoid of the flow of $\pd{\theta} +\lambda \pd{y}$. By the above characterization of $\CK$, we have $\CK=\ZZ\times P$, which is an embedded Lie subgroupoid of $\CH(\cF)$. Hence $\CK\backslash \CH(\cF)=(\RR/\ZZ)\times P$, and quotienting by the  groupoid action $\theta$ we obtain  $(\RR/\ZZ)\times S^1$. One checks easily that the induced groupoid structure is given by the transformation groupoid of the Lie group $\RR/\ZZ$ acting by rotations on  $S^1$, which is isomorphic to the pair groupoid structure on $S^1\times S^1$. Prop. \ref{prop:norsubG} and Rem. \ref{rem:hypo} state that it is isomorphic to $\CH(\cF_M)$.}
\end{ex}

\appendix
\section{Appendix}
 
\subsection{{A lemma about generating sets}}

We  prove the following statement, {which will be used in Appendix \ref{app:thmXi}}. 
{Recall that $\widehat{\cF}$ denotes the global hull of $\cF$, see Def. \ref{defi:globalhull}.}

\begin{lem}\label{lem:projgen2}
	Let $\pi : P \to M$ be a {surjective} submersion with connected fibers. Let $\cF$ be a singular foliation on $P$, such that $[\Gamma_c(\ker d\pi),\CF] \subset \Gamma_c(\ker d\pi)+\CF$. 
	
	i) The set
	\begin{align*}
	& {\widehat{\cF}}^{proj}:=\{X\in {\widehat{\cF}}: X\text{ is $\pi$-projectable to a vector field on $M$}\}  
	\end{align*}
  generates $\cF$ as a ${C^{\infty}_c(P)}$-module. 

ii) The singular foliation $\CF_M$  on $M$ introduced in Thm. \ref{thm:sur.hol} admits the following description:
$$\cF_M=\pi_*(\cF):=Span_{{C_c^{\infty}(M)}}\{\pi_*X : X\in {\widehat{\cF}}^{proj}\}$$
\end{lem}

\begin{proof}
	
We first make a claim.

\noindent{\bf Claim:} \emph{{Lemma \ref{lem:projgen2} holds in the special case that $\Gamma_c(\ker d\pi)\subset \CF$.}}

\noindent{Indeed, in this special case, by Prop. \ref{prop:submfol} there is  a unique singular foliation  $\cF_M$ on $M$ with $\pi^{-1}(\cF_M) = \cF$.
Given this, i) is a consequence of Definition \ref{def:pullback}. For ii), note that $\pi^{-1}(\pi_*(\cF))=\CF$, as can be checked using i). Since} $\CF=\pi^{-1}\CF_M$, we obtain
{$\cF_M=\pi_*(\cF)$ by the uniqueness statement in Prop. \ref{prop:submfol}. This proves the claim.} 

\smallskip
	Take $\CF^{\text{big}}:=\Gamma_c(\ker d\pi)+\CF$, a singular foliation satisfying the condition of the above claim.	{We now proceed to prove the two items of the lemma.}
	
i)  {By the claim}, 	 ${\widehat{\cF^{\text{big}}}}^{proj}$
 generates $\cF^{\text{big}}$ as a ${C^{\infty}_c(P)}$-module.
 Take $X\in \CF\subset \CF^{\text{big}}$.
There exist finitely many $Y_j\in {\widehat{\cF^{\text{big}}}}^{proj}$ and $f_j\in \CI_c(P)$ such that
$X= \sum_j f_j Y_j$.
	By definition of $\CF^{\text{big}}$, we can write $Y_j= \widehat{Y}_j+Z_j$ with $\widehat{Y}_j\in {\widehat{\cF}}^{proj}$  and $Z_j\in \Gamma(Ker(d\pi))$. Then:
	\[X=\sum_j {f_j} Y_j= \sum_j {f_j} \widehat{Y}_j + \sum_j {f_j} Z_j. \]
The last term $\sum_j {f_j} Z_j=X-\sum_j {f_j} \widehat{Y}_j$   lies in {$\cF$ as the difference of two elements of $\cF$, and is $\pi$-projectable (to the zero vector field on $M$). Hence this last term lies} in ${\widehat{\cF}}^{proj}$, and we have proven i).

ii) {We have $\pi^{-1} (\pi_* (\CF))= \pi^{-1}( \pi_* (\CF^\text{big}))=\CF^\text{big}$, using the claim in the second equality, and $\CF^\text{big}=\pi^{-1}(\CF_M)$ by definition.}    The uniqueness in Proposition \ref{prop:submfol} implies that $\CF_M =\pi_* (\CF)$.
\end{proof}

\subsection{Proof of Thm. \ref{thm:sur.hol}}\label{app:thmXi}

The following proposition is a special case\footnote{This special case is obtained from \cite[Prop. D.4]{SingSub}  taking $\cB_1$ to be the singular foliation $\cF$, $G_i$ to be the pair groupoid $M_i\times M_i$ and $F:=\pi\times\pi$.} 
 of \cite[Prop. D.4]{SingSub}, {augmented with the statement that  $\Xi$ is an open map.}

\begin{prop}\label{prop:morphfol}
Let  $\pi \colon P\to M$ be a surjective submersion. 

Let $\cF$ be a singular foliation on $P$, and assume that it satisfies the following condition:
\begin{align}\label{star}
& {\widehat{\cF}}^{proj}:=\{X\in {\widehat{\cF}}: X\text{ is $\pi$-projectable to a vector field on $M$}\}\\
&\text{generates $\cF$ as a ${C^{\infty}_c(P)}$-module.}\nonumber
\end{align}
Let  $\pi_*(\cF):=Span_{{C_c^{\infty}(M)}}\{\pi_*X : X\in {\widehat{\cF}}^{proj}\}$, which is a singular foliation on $M$.

Then there is a canonical, {open,} surjective morphism of topological groupoids $\Xi \colon \Ho(\cF)\to {\Ho(\pi_*(\cF))}$ covering $\pi$.\end{prop}

\begin{rem}\label{rem:appsub}
{If $\CH(\CF)$ and $\CH(\CF_M)$ are Lie groupoids, then $\Xi$ is a submersion. This follows from the proof of Prop. \ref{prop:morphfol} given below, since in that case the quotient maps from bisubmersions to holonomy groupoids are submersive.}
\end{rem}

 {We can now prove Thm. \ref{thm:sur.hol}:}

\begin{proof}[Proof of Thm. \ref{thm:sur.hol}]
Apply Prop. \ref{prop:morphfol}. Notice that 
condition \eqref{star} is satisfied by Lemma \ref{lem:projgen2} i) 
and $\pi_*\cF=\cF_M$ {by Lemma \ref{lem:projgen2} ii)}.
\end{proof}

{In order to keep this paper self-contained,    we now provide a proof for  Prop. \ref{prop:morphfol}.  It differs from the proof found in \cite[Prop. D.4]{SingSub}, in that it  allows to write down the morphism $\Xi$ more explicitly, and makes clear that $\Xi$ is an open map.}  We start with a lemma.

\begin{lem}\label{prop:carry.diffM} Let $\pi:P\fto M$ and $\cF$ be as in Prop. \ref{prop:morphfol}, and { $\cF_M :=\pi_*(\cF)$}.  Then there exists a family of path holonomy bisubmersions $\cS$ for $\cF$, {so that $\cup_{U\in \cS}\bs(U)=P$,} with these properties:  

{\begin{itemize}

	\item[i)]  For any $U\in \cS$ we have that $(U,\pi \circ\bt, \pi\circ\bs)$ is a bisubmersion for $\CF_M$. Further it is adapted to a path holonomy bisubmersions for $\cF_M$.
	
	\item[ii)] 
{	Let $\CU$ be the atlas generated by $\CS$. Then for any $U\in \CU$ we have that $(U,\pi \circ\bt, \pi\circ\bs)$ is a  (source connected) bisubmersion for $\CF_M$. Further
	the atlas generated by $$\pi\CU:=\{(U,\pi\circ \bt,\pi\circ\bs) \st U\in\CU\},$$
is equivalent to a path-holonomy atlas for $\CF_M$.} 
\end{itemize}} 
 \end{lem}

\begin{proof}
 
\noindent{\bf Claim:} \emph{$\cF$ is locally generated by {\emph{finitely many}} $\pi$-projectable vector-fields in $\widehat{\cF}$.}

Indeed, for any $p\in P$ there is a neighborhood $V_0\subset P$ and finitely many generators ${Y_1,\dots,Y_k}\in \CX(V_0)$  of $\iota_{V_0}^{-1}(\cF)$,  for $\iota_{V_0}$ the inclusion. Take any precompact open set  $V\subset V_0$ containing $p$ and $\rho_V\in \CI_c(V_0)$ such that $\rho_V=1$ on $V$. {For each $i$, since $\rho_V Y_i\in \cF$, condition \eqref{star} assures that} there is a finite number of $\pi$-projectable elements $X^j_i\in {\widehat{\cF}}^{proj}$
{and $f^j_i\in \CI_c(P)$}
		 such that:
		\[\rho_V Y_i=\sum_j {f^j_i} X^j_i.\]
{Therefore} every element of $\iota^{-1}_V(\cF)$ is a $\CI_c(V)$-linear combination of the $X^j_i$, which are $\pi$-projectable {and lie in $\widehat{\cF}$}. {This proves the claim.}

{i)} Now, for every point $p_0$ of $P$, take a minimal set of $\pi$-projectable elements
$\{X_1,\dots,X_n\}$ in $\widehat{\cF}$ that are local generators of $\cF$ nearby that point.
 Let $(U,\bt,\bs)$ be the corresponding  {path-holonomy}  bisubmersion, where $U\subset \RR^n\times P$. Then 
 $(U,\pi\circ\bt,\pi\circ\bs)$ is a bisubmersion for $\cF_M$,
 with source map $(\lambda,p)\mapsto \pi(p)$ and target map $(\lambda,p)\mapsto 
 \exp_{\pi(p)}(\sum \lambda_i \pi_*X_i).$
 A way to see this is to apply \cite[Lemma 2.3]{AndrSk} to the {path-holonomy}  bisubmersion $W   \subset {\RR^n\times M}$ for $\cF_M$ corresponding to the generators $\{\pi_*X_1,\dots,\pi_*X_n\}$ and to the submersion\footnote{{More precisely, to its  restriction to  $(Id_{\RR^n},\pi)^{-1}(W)\cap U$.}} $(Id_{\RR^n},\pi)\colon \RR^n\times P\to \RR^n\times M$.  We observe that $(Id_{\RR^n},\pi)$ is a morphism of bisubmersions from $(U,\pi\circ\bt,\pi\circ\bs)$  to $W$. This shows that the former bisubmersion is adapted (see Def. \ref{def:atlas.bi}) to the latter.

 {ii) Take $U\in \CU$, without loss of generality assume $U= U_1\circ \dots \circ U_k$ for $U_i\in \CS$. Denote $\pi U_i :=(U_i, \pi\circ \bt,\pi\circ\bs)$, which are bisubmersions for $\CF_M$  by i) . Note that the inclusion map $\omega\colon	U_1\circ \dots \circ U_k\fto \pi U_1\circ \dots \circ \pi U_k$ makes the following diagram commute:} 
 \begin{equation}\label{diag:piu}
\begin{tikzcd}
	U \ar[d, shift right=.2em, swap, "\bt"]\ar[d,shift left=.2em,"\bs"] \ar[r,"\omega"] & \pi U_1\circ \dots \circ \pi U_k \ar[d, shift right=.2em, swap,"\widetilde{\bt}"]\ar[d,shift left=.2em,"\widetilde{\bs}"] \\
	P \ar[r,"\pi"] & M
	\end{tikzcd} 
\end{equation}
{Because of this commutative diagram and since $\pi U_1\circ \dots \circ \pi U_k$ is a bisubmersion for $\CF_M$ one gets that
\begin{equation}\label{eq:igual1}
A:= (\pi\circ \bt)^{-1}\CF_M=(\pi\circ \bs)^{-1}\CF_M.
\end{equation}
}

{The l.h.s. is 
\begin{align*}
(\pi\circ \bt)^{-1}\CF_M=\bt^{-1}(\CF+\ker_c(d\pi))&=\bt^{-1}(\CF)+\bt^{-1}(\ker_c(d\pi))\\&=\ker_c(d\bs)+\ker_c(d\bt)+\bt^{-1}(\ker_c(d\pi))\\
&=\ker_c(d\bs)+ \bt^{-1}(\ker_c(d\pi)),
\end{align*}
using respectively that $\pi^{-1}\CF_M=\CF+\ker_c(d\pi)$,
that $\bt$ is a submersion, and that   $U$ is a bisubmersion for $\CF$. Here we use the short-hand notation $\ker_c(d\pi):=\Gamma_c(\ker(d\pi))$. Repeating for the r.h.s., from eq. \eqref{eq:igual1} we obtain
\[A=\ker_c(d\bs) +\bt^{-1}(\ker_c(d\pi))= \ker_c(d\bt)+\bs^{-1}(\ker_c(d\pi)).\]
This implies that 
\[A= \bt^{-1}(\ker_c(d\pi))+\bs^{-1}(\ker_c(d\pi))= \ker_c(d(\pi\circ\bt))+\ker_c(d(\pi\circ\bs)),\]
i.e. that $(U,\pi\circ\bt,\pi\circ\bs)$ is a bisubmersion for $\CF_M$. 
}

{
By construction, the bisubmersion $\pi U_1\circ \dots \circ \pi U_k$ lies in   the atlas for $\CF_M$ generated by $\pi\CS$.
Thus the commutative diagram \eqref{diag:piu} shows that $(U,\pi\circ\bt,\pi\circ\bs)$ is adapted to the atlas for $\CF_M$ generated by $\pi\CS$, which in turn  is adapted to 
  a path-holonomy atlas by i). This shows that the atlas generated by $\pi\CU$ is adapted to 
  a path-holonomy atlas for $\CF_M$. It is actually equivalent to such an atlas, because a path-holonomy atlas is adapted to any other atlas.
}

\end{proof}

\begin{rem}
	Not every bisubmersion $(U,\bt,\bs)$ for $\cF$ satisfies that $(U,\pi\circ\bt,\pi\circ\bs)$ is a bisubmersion for $\cF_M$. For instance take $P:=\RR^2$, $\cF=0$ and $M=\RR$ with map $\pi\colon P\fto M$ given by the first projection. Then $\cF_M=0$. Now take any diffeomorphism $\phi\colon P\fto P$  that does not preserve the foliation $\pi^{-1}(\cF_M)$ 
  by the fibers of $\pi$. Then $(P,Id, \phi)$ is a bisubmersion for $\cF$ but $(P,\pi,\pi\circ \phi)$ is not a bisubmersion for $\cF_M$.
\end{rem}

\begin{proof}[Proof of Prop. \ref{prop:morphfol}:] 
{
{Let $\CU$ be an atlas for $\CF$ as in Lemma \ref{prop:carry.diffM}. For every $U\in \CU$ we use the   short-hand notation $\pi U$ to denote 
  $(U,\pi \circ\bt, \pi\circ\bs)$, a bisubmersion for $\CF_M$.
Define $\Xi$ as follows:}
\begin{equation}\label{eq:Xidescr2}
\Xi \colon \CH(\CF)\fto \CH(\CF_M);	[u]\mapsto [u]_M,
\end{equation}
where {$u\in U \in \CU$, and}
$[u]_M$ is the class of $u$ {seen as an element of} $\pi U\in \pi \CU $.  {
Here we used that, by Lemma \ref{prop:carry.diffM}, $\CH(\CF_M)$ agrees with the groupoid associated to the atlas generated by $\pi \CU$.}
}

{
The map $\Xi$ is well defined. If $u_1\in U_1\in \CU$ and $u_2\in U_2\in \CU$ are equivalent, then there exists a morphism of bisubmersions sending $u_1$ to $u_2$. Using the same morphism it is clear that $u_1\in \pi U_1\in \pi\CU$ is equivalent to $u_2\in \pi U_2\in \pi \CU$, {i.e. that $[u_1]_M=[u_2]_M$}. {The same argument shows that $\Xi$ is independent of the specific choice of the atlas $\CU$, thus canonical.}
}

{
 It is   clear that $\Xi$ covers $\pi$ and sends the identity bisection of $\CH(\CF)$ to the identity bisection of $\CH(\CF_M)$.
To prove that it is a morphism of set theoretic groupoids we only need to prove that it preserves the composition. It does because for any $U_1,U_2\in \CU$, the inclusion map $\pi(U_1\circ U_2)\fto \pi U_1\circ \pi U_2$ 
is a morphism of bisubmersions for $\CF_M$.
}

{
We  check that $\Xi$ is a continuous open map. 
This holds 
because in the following commutative diagram the quotient maps $Q$ and $Q_M$ are continuous and   open  (see \cite[Lemma 3.1]{MEsingfol}), and because $Q$ is surjective.}

 \begin{equation*} 
\begin{tikzcd}
	\sqcup_{U\in \CU}U \ar[d,   "Q"]  \ar[r,"Id"] & \sqcup_{U\in \CU} \pi U \ar[   "Q_M",d] \\
\CH(\CF) \ar[r,"\Xi"] & \CH(\CF_M)
	\end{tikzcd} 
\end{equation*}

{
The map $\Xi$ is surjective. By the above, $\Xi(\CH(\CF))$ is a neighborhood of the identities of $\CH(\CF_M)$. Because $\Xi$ is a morphism of topological groupoids, $\Xi(\CH(\CF))$ is a symmetric set closed under compositions. It is well known that any $\bs$-connected topological groupoid is generated by any symmetric neighborhood of the identities \cite{MK2}. Because $\CH(\CF_M)$ is $\bs$-connected we obtain $\Xi(\CH(\CF))=\CH(\CF_M)$, i.e. $\Xi$ is surjective.
}
\end{proof}

  {The following corollary extends the conclusions of Lemma \ref{prop:carry.diffM} ii) to arbitrary source connected atlases (as defined in Def. \ref{def:scon.bisub}).}
\begin{cor}\label{cor:adpt}
	 {Let $\pi:P\fto M$, $\CF\subset \CX_c(P)$ and $\CF_M\subset\CX_c(M)$ as in Prop. \ref{prop:morphfol}. Then for any source connected atlas $\CU'$  for $\CF$ we have that $\pi\CU':=\{\pi U':U'\in \CU'\}$ is an atlas equivalent  to a path holonomy atlas for $\CF_M$.}  
\end{cor}
\begin{proof}
  {We first observe that in Lemma \ref{prop:carry.diffM}, actually $\pi\CU$ is already an atlas. This follows from the fact that the map $\Xi$, given as in eq. \eqref{eq:Xidescr2}, is surjective.}

  {Now let $\CU'$ be any source connected atlas for $\CF$. Then $\CU'$ is adapted to $\CU$ (See \cite[\S 4]{T.ALF}). This means that for any element $u\in U'\in \CU'$ 
there exists a morphism of bisubmersions $\omega_u$ from a 
neighborhood $U'_u\subset U'$ to a bisubmersion $U\in \CU$. In particular the following diagram commutes:} 
	\[\begin{tikzcd}
	U'_u \ar[d, shift right=.2em, swap, "\bt"]\ar[d,shift left=.2em,"\bs"] \ar[r,"\omega_u"] & U \ar[d, shift right=.2em, swap,"\pi\circ\bt"]\ar[d,shift left=.2em,"\pi\circ\bs"] \\
	P \ar[r,"\pi"] & M
	\end{tikzcd}\]
	{The triple $(U'_u,\bt\circ\pi,\bs\circ\pi)$ is a bisubmersion for $\CF_M$, since the argument following diagram \eqref{diag:piu} can be applied identically to the diagram above. Moreover, since being a bisubmersion is a local property we have that $\pi U':=(U',\bt\circ\pi,\bs\circ\pi)$ is a bisubmersion for $\CF_M$.}
	
{The families $\pi\CU':=\{\pi U':U'\in \CU'\}$ and $\pi \CU$ are adapted to each other, since the atlases $\CU'$ and  $\CU$ are equivalent. This implies that $\pi\CU'$ is already an atlas for $\CF_M$, equivalent to $\pi \CU$. The latter is equivalent to a path holonomy atlas by Lemma \ref{prop:carry.diffM} ii), so we are done.}
\end{proof}

\bibliographystyle{alpha}
\bibliography{Lie2bib}

\end{document}